\documentclass[reqno,12pt]{amsart}
\usepackage{geometry}
\usepackage{amsmath,amssymb,mathrsfs,color,stackengine}
\usepackage[colorlinks,
linkcolor=red,
anchorcolor=green,
citecolor=blue,
]{hyperref}
\usepackage{soul,color,xcolor}
\usepackage[upint]{stix}
\usepackage{subfigure}
\usepackage{float}
\usepackage{times}
\usepackage{tikz}
\usetikzlibrary{intersections}
\usepackage{paralist}

\usepackage{comment}

\makeatletter
\def\setliststart#1{\setcounter{\@listctr}{#1}%
  \addtocounter{\@listctr}{-1}}
\makeatother

\makeatletter
\@addtoreset{figure}{section}
\makeatother

\usepackage{calc}
 \newtheorem{The}{Theorem}[section]
 \newtheorem{Cor}[The]{Corollary}
 \newtheorem{Lem}[The]{Lemma}
 \newtheorem{Pro}[The]{Proposition}
 \theoremstyle{definition}
 \newtheorem{defn}[The]{Definition}
 \newtheorem{Rem}[The]{Remark}
 
 \numberwithin{equation}{section}

\newcommand{\R}{\mathbb{R}}

\title{A geometric approach to Mather quotient problem} 
\author{Wei Cheng}
\address{Wei Cheng, School of Mathematics, Nanjing University, 22 Hankou Road, Nanjing, 210093, Jiangsu, China.}
\email{chengwei@nju.edu.cn}
\author{ Wenxue Wei}
\address{Wenxue Wei, School of Mathematics, Nanjing University, 22 Hankou Road, Nanjing, 210093, Jiangsu, China.}
\email{wwx3708@gmail.com}
\date{\today}
\begin{document}
\maketitle

\begin{abstract}
Let $(M,g)$ be a closed, connected and orientable Riemannian manifold with nonnegative Ricci curvature. Consider a Lagrangian $L(x,v):TM\to\R$ defined by $L(x,v):=\frac 12g_x(v,v)-\omega(v)+c$, where $c\in\R$ and $\omega$ is a closed 1-form. From the perspective of differential geometry, we estimate the Laplacian of the weak KAM solution $u$ to the associated Hamilton-Jacobi equation $H(x,du)=c[L]$ in the barrier sense. This analysis enables us to prove that each weak KAM solution $u$ is constant if and only if $\omega$ is a harmonic 1-form. Furthermore, we explore several applications to the Mather quotient and Ma\~n\'e's Lagrangian.
\end{abstract}

\section{Introduction}\label{sec1}

This paper focuses on a significant problem in Aubry-Mather theory originally posed by John Mather, concerning the Mather quotient. The Aubry-Mather theory employs variational methods to study Hamiltonian dynamical systems. Mather developed a theory to study the dynamics of the associated Euler-Lagrangian flow in frame of Tonelli theory in calculus of variations, by introduced certain invariant sets of the global Lagrangian dynamical systems such as Aubry set, Mather set, etc. (as detailed in \cite{Mather1991,Mather1993} and further elaborated in \cite{Mane1992}). 

\subsection{Mather quotient}

Concentrating on time-independent case, we suppose $M$ is a closed and connected smooth manifold with $TM$ and $T^*M$ its tangent and cotangent bundle respectively. A function $L(x,v):TM\to\R$ is called a \emph{Tonelli Lagrangian} if $L$ is of class $C^r$ ($r\geqslant3$) and $L(x,\cdot)$ is strictly convex and uniformly superlinear on $T_xM$ for all $x\in M$. The \emph{Tonelli Hamiltonian} $H:T^*M\to\R$ associated to a Tonelli Lagrangian $L$ is defined by $H(x,p)=\sup_{v\in T_xM}\{p(v)-L(x,v)\}$, $(x,p)\in T^*M$. In \cite{Mather1993} Mather introduced the \emph{Peierls' barrier function} $h:M\times M\to\R$, $h(x,y)=\liminf_{t\to+\infty}\{A_t(x,y)+c[L]t\}$, $x,y\in M$. Here $A_t(x,y)=\inf_{\rho}\int^t_0L(\rho,\dot{\rho})\ ds$ where the infimum is taken over the family of absolutely continuous curve $\rho:[0,t]\to M$ connecting $x=\rho(0)$ to $y=\rho(t)$, and $c[L]\in\R$, the Ma\~n\'e's critical value, is the unique constant such that $h$ is finite-valued (see \cite{Fathi_Siconolfi2004}). The projected Aubry set is defined by $\mathcal{A}(L)=\{x\in M: h(x,x)=0\}$. In \cite{Mather1993}, Mather also introduced a pseudo-metric $\delta$ on $\mathcal{A}(L)$ by
\begin{align*}
	\delta(x,y)=h(x,y)+h(y,x).
\end{align*}
The relation $x\sim y\Leftrightarrow\delta(x,y)=0$ gives an equivalence relation on $\mathcal{A}(L)$. The associated quotient space $(\mathcal{A}(L),\sim, \delta)$ is the so-called \emph{Mather quotient}.

In \cite{Mather2003}, Mather showed that if $M$ has dimension $2$ or if the Lagrangian is the kinetic energy associated to a Riemannian metric on $M$ with $\dim M\leqslant3$, then the Mather quotient is totally disconnected, i.e. every connected component consists of a single point. Unfortunately, this does not hold in higher dimensions (see \cite{Burago_Ivanov_Kleiner1997,Mather2004}). The totally  disconnectedness of Mather quotient is closely related to the upper semi-continuity of the Aubry set and has been studied in several earlier works such as \cite{Bernard2010b,Fathi_Figalli_Rifford2009,Mather2003,Sorrentino2008}. Those works consider this problem from either topological or variational points of view. 

Certain Morse-Sard type results on this problem can be found in \cite{Fathi_Figalli_Rifford2009}. The authors proved

\begin{The}
	Let $L$ be a Tonelli Lagrangian on a closed smooth manifold $M$. Then, it satisfies the Mather disconnectedness condition (i.e. for every pair $u_1,u_2$ of weak KAM solutions, the image $(u_1-u_2)(\mathcal{A}(L))\subset\R$ is totally disconnected) in the following five cases:
	\begin{enumerate}[\rm (i)]
		\item The dimension of $M$ is 1 or 2 .
		\item The dimension of $M$ is 3, and $\tilde{\mathcal{A}}(L)$, the Aubry set\footnote{In the context of weak KAM theory, $\tilde{\mathcal{A}}(L)=\bigcap_{u}\{(x,v)\in TM: d_xu=L_v(x,v)\}\subset TM$ with $u$ taken over all $C^1$ subsolution of \eqref{eq:HJ_wk}}, contains no fixed point of the associated Euler-Lagrange flow $\Phi_t^L$ (which is defined in Section 2).
		\item The dimension of $M$ is 3, and $L$ is of class $C^{3,1}$.
		\item The Lagrangian is of class $C^{k,1}$, with $k \geqslant 2\dim M-3$, and every point of $\tilde{\mathcal{A}}(L)$ is fixed under the Euler-Lagrange flow $\Phi_t^L$ .
		\item The Lagrangian is of class $C^{k, 1}$, with $k \geqslant 8 \dim M-8$, and either each point of $\tilde{\mathcal{A}}(L)$,  is fixed under the Euler-Lagrange flow $\Phi_t^L$ or its orbit in the $\tilde{\mathcal{A}}(L)$ is periodic with strictly positive period.
	\end{enumerate}
\end{The}

In \cite{Bernard2010b}, Bernard listed several so-called \emph{coincidence hypothesis} and shown that the Mather disconnectedness condition implies the property that the Mather quotient is totally disconnected. He also obtained the upper semi-continuity of the Aubry set under these conditions.

In \cite{Sorrentino2008} the author proved

\begin{The}
	Let $M$ be a closed connected smooth  manifold with dimension $n\geqslant 1$  and let $L$ be a Tonelli Lagrangian such that
	\begin{align*}
		\Lambda_L:=\{(x, L_v(x,0)): x\in M\}
	\end{align*}
	is a Lagrangian submanifold of $T^*M$ and $L(x,0)\in C^r(M)$,  with $r\geqslant 2n-2$ and $L_v(x,0)  \in C^2(M)$. Then, for every $[\omega]$ in the Liouville class of $\Lambda_L$ and $L_\omega(x,v):=L(x,v)-\omega(v)$, the Mather quotient $(\mathcal{A}(L_\omega),\sim,\delta)$ is totally disconnected.
\end{The}

From late 1990's, Albert Fathi developed celebrated weak KAM theory which serves as a conceptual bridge between the Aubry-Mather theory and the realm of partial differential equations (PDEs). Fathi proved there exists a unique constant $c[L]$, exactly the Ma\~n\'e's critical value, such that the Hamilton-Jacobi equation
\begin{equation}\label{eq:HJ_wk}\tag{HJ}
	H(x,d_xu)=c[L],\qquad x\in M
\end{equation}
admits a weak solution $u$ which is a common fixed point of the Lax-Oleinik semigroup $T^-_t+c[L]t$ for $t\geqslant0$ (See more details in Section 2). Such weak solutions are called weak KAM solutions. Weak KAM theory enables the application of PDEs and tools from differential geometry.

If $X$ is a $C^k$ vector field on a Riemannian manifold $(M,g)$ with $k\geqslant 2$, introduced by Ricardo Ma\~n\'e in \cite{Mane1992}, the Ma\~n\'e Lagrangian $L_X:TM\to\R$ associated to $X$ is defined by
\begin{align*}
	L_X(x,v)=\frac 12g_x(v-X,v-X),\quad \forall (x,v)\in TM. 
\end{align*}
In \cite{Fathi_Figalli_Rifford2009}, the authors also obtained 

\begin{Pro}
	Let $L_X:TM\to\R$ be the Ma\~n\'e Lagrangian associated to a $C^k$ vector field $X$ on a closed connected Riemannian manifold $(M,g)$ with $k\geqslant 2$. Assume that $L_X$ satisfies the Mather disconnectedness condition. Then we have the following:
	\begin{enumerate}[\rm (i)]
		\item The projected Aubry set $\mathcal{A}(L_X)$ is the set of chain-recurrent points of the flow of $X$ on $M$.
		\item The constants are the only weak KAM solutions of \eqref{eq:HJ_wk} associated to $L_X	$ if and only if every point of $M$ is chain-recurrent under the flow of $X$.
	\end{enumerate}
\end{Pro}

\begin{The}
	Let $X$ be a $C^k$ vector field, with $k\geqslant 2$, on a closed connected Riemannian manifold $(M,g)$. Assume that one of the following conditions hold:
	\begin{enumerate}[\rm (i)]
		\item The dimension of $M$ is 1 or 2.
		\item The dimension of $M$ is 3, and the vector field $X$ never vanishes.
		\item The dimension of $M$ is 3, and $X$ is of class $C^{3,1}$.
	\end{enumerate}
	Then the projected Aubry set $\mathcal{A}(L_X)$ of the Ma\~n\'e Lagrangian $L_X:TM\to\R$ associated to $X$ is the set of chain-recurrent points of the flow of $X$ on $M$. Moreover, the constants are the only weak KAM solutions of \eqref{eq:HJ_wk}  associated to  $L_X$ if and only if every point of $M$ is chain-recurrent under the flow of $X$.
\end{The}

In a recent work (\cite{Contreras_Miranda2020}), applying Bernard-Contreras's theorem (\cite{Bernard_Contreras2008}) the authors proved that there exists a residual subset $\mathcal{G}\subset C^\infty(M)$  such that, if $L(x,v)=\frac{1}{2} e^{f(x)}g_x(v,v)$ with $f\in\mathcal{G}$, then, for any $[\omega]\in H^1(M,\R)$ and $L_\omega(x,v):=L(x,v)-\omega(v)$, the Mather quotient $(\mathcal{A}(L_\omega) ,\sim, \delta)$ has a finite number of elements.

\subsection{Mather quotient, Ricci curvature and Harmonic 1-form}

In this paper, we adopt a novel geometric perspective to examine Mather's problem, with a particular focus on the Ricci curvature of the kinetic Riemannian metric. A central objective of our research is to provide an estimation of the Laplacian of the weak KAM solution in relation to the kinetic energy Lagrangian. This Laplacian estimation is intrinsically linked to the core aspects of Mather's problem.

In his seminal work \cite{Mather1991}, John Mather observed the invariance of the Euler-Lagrange flow under transformations induced by adding exact 1-forms, and noted that the Aubry set is determined solely by the de Rham cohomology class. Furthermore, leveraging  Hodge's theorem, we understand that on a compact, oriented, smooth manifold, the Hodge cohomology $\mathcal{H}^1(M,\R)$ is isomorphic to the de Rham cohomology $H^1(M,\R)$. This isomorphism enables us to employ the Hodge cohomology to delve into the rigidity properties of the Aubry set, particularly in the context of the kinetic energy Lagrangian, under the condition that the manifold $M$ possesses nonnegative Ricci curvature.

The method used in this paper draw inspiration from the celebrating splitting theorem of Gromoll and Cheeger. In the realm of differential geometry, a particularly effective approach to estimating the Laplacian of solutions to the Hamilton-Jacobi equation involves the utilization of the Riccati equation. Through the viewpoint of differential geometry, we provide several estimates for the Laplacian of the weak KAM solution of \eqref{eq:HJ_wk}. Furthermore, we establish the following results.

\begin{The}\label{mechnaical case}
	Suppose $(M, g)$ is a closed connected Riemannian manifold with nonnegative Ricci curvature. Then, for each weak KAM solution $u$ of \eqref{eq:HJ_wk} associated to the mechanical Lagrangian $L(x, v)=\frac{1}{2} g_x(v, v)+f(x)$, we have
	\begin{align*}
		\Delta u(x)\leqslant\sqrt{-nk},\qquad \forall x\in M
	\end{align*}
	in the barrier sense, where $n=\dim M \geqslant2$  and $k$ is some nonpositive number such that $\Delta f(x)\geqslant k$ for all $x\in M$.
\end{The}

\begin{The}\label{solution and divergence}
	Suppose  $L(x,v)=\frac{1}{2}g_x(v,v)-f(x)-\omega(v)$ is  a Tonelli  Lagrangian on a closed connected Riemannian manifold $(M,g)$ with $ \omega$ is a closed 1-form. Set
	\begin{align*}
		E=\{(x,v)\in TM: H(x,L_v(x,v))=c[L]\},
	\end{align*}
	where $H:T^*M\rightarrow \R$ is the Tonelli Hamiltonian associated to $L$. If
	\begin{align*}
		(x,v)\in E\qquad\Longrightarrow\qquad\operatorname{Ric}(\dot{\rho}(s))+\Delta f(\rho(s))\geqslant0
	\end{align*}
	where $\rho(s)=\pi\circ\Phi_s^L(x,v)$, $s\in(-\infty,0]$, then, for any weak KAM solution $u$ of  \eqref{eq:HJ_wk}, we have
	\begin{align*}
		\Delta u(x)\leqslant-\operatorname{div}\omega^{\#}(x),\quad\forall x\in M
	\end{align*}
	in the barrier sense.
\end{The}

The theorems above implies some consequences on the Mather quotient and Ma\~n\'e Lagrangian.

\begin{The}\label{constant if and only if }
	Given an orientable connected closed Riemannian manifold $(M, g)$ with nonnegative Ricci curvature. Let $\omega$ be a closed 1-form on $M$ and let $X:=\omega^{\#}$ be its corresponding vector field. Then, for every constant $c\in\R$, each weak KAM solution $u$ of  \eqref{eq:HJ_wk} associated to the Lagrangian $L(x,v):=\frac 12g_x(v,v)-\omega(v)+c$ is  constant if and only if $\omega$ is a harmonic 1-form.
\end{The}

\begin{Cor}\label{mather quotient}
	Given an orientable connected closed Riemannian manifold $(M,g)$ with nonnegative Ricci curvature. Let $L(x,v)=\frac 12g_x(v,v)$ be the kinetic energy associated to the Riemannian metric. Then, for each $[\omega]\in H^1(M,\R)$ and $L_\omega(x,v):=L(x,v)-\omega(v)$, the projected Aubry set $\mathcal{A}(L_\omega)=M$, and the set-valued map ${H}^1(M,\R)\ni[\omega]\rightrightarrows\mathcal{A}(L_\omega)$ is constant. Moreover, the Mather quotient $(\mathcal{A}(L_\omega),\sim, \delta)$ associated to the Lagrangian $L_\omega$ is a singleton.
\end{Cor}

\begin{The}\label{mane constant}
	Given an orientable connected closed Riemannian manifold $(M,g)$. Let $L(x,v)=\frac 12g_x(v-\omega^\sharp,v-\omega^\sharp)$ be the Ma\~n\'e's Lagrangian with $\omega$ a closed 1-form and $\omega^{\#}$ its corresponding vector field. Set
	\begin{align*}
		E=\{(x,v)\in TM: H(x,L_v(x,v))=c[L]\},\qquad f(x)=\frac{1}{2}g_x(\omega^\sharp,\omega^\sharp),
	\end{align*}
	with $H(x,p)$ the associated Tonelli Hamiltonian. If the following condition holds
	\begin{align*}
		(x,v)\in E\qquad\Longrightarrow\qquad\operatorname{Ric}(\dot{\rho}(s))+\Delta f(\rho(s))\geqslant0
	\end{align*}
	where $\rho(s)=\pi\circ\Phi_s^L(x,v)$, $s\in(-\infty,0]$, then every weak KAM solution $u$ of  \eqref{eq:HJ_wk} is constant if and only if $\omega$ is a harmonic 1-form, and the Mather quotient $(\mathcal{A}(L),\sim, \delta)$ is a singleton if and only if $ \omega$ is a harmonic 1-form. In particular, if $(M,g)$ has nonnegative Ricci curvature and $\omega$ is a harmonic 1-form, each weak KAM solution $u$ of  \eqref{eq:HJ_wk} is constant.
\end{The}

The paper is organized as follows: In Section 2, we review certain basic facts from Aubry-Mather theory and Riemannian geometry, with a particular emphasis on the characteristics of conjugate points.  Section 3 is dedicated to the Riccati equation. In Section 4 we prove the main results of this paper. The paper also includes two appendices: one provides the proofs for the points mentioned in Section 2, and the other discusses the index form in the context of the Lagrangian framework.


\section{Preliminaries and Notions}

\subsection{Facts from Aubry-Mather theory and weak KAM theory}

We now recall the basic facts from Aubry-Mather theory and weak KAM theory (see \cite{Fathi_book,Fathi_Siconolfi2004,Fathi_Figalli2010} and more details on semiconcavity in \cite{Cannarsa_Cheng3,Cannarsa_Sinestrari_book}).

If $L$ is a Tonelli Lagrangian, we define the \emph{generating function}
\begin{align*}
	A_t(x, y)=\inf _{\rho \in \Gamma_{x, y}^t} \int_0^t L(\rho(s), \dot{\rho}(s))\ ds,\quad t>0, x,y\in M
\end{align*}
where $\Gamma_{x, y}^t=\{\rho\in AC([0, t],M): \rho(0)=x, \rho(t)=y\}$. A \emph{minimal curve for $A_t(x, y)$} is an absolutely continuous curve $\rho\in\Gamma^t_{x,y}$ such that
\begin{align*}
	A_t(x, y)=\int_0^t L(\rho(s),\dot{\rho}(s))ds.
\end{align*}
By classical Tonelli theory, the infimum in the definition of $A_t(x,y)$ can be achieved and any minimal curve $\rho$ is as smooth as $L$. In local charts, $\rho$ satisfies Euler-Lagrange equation
\begin{align}\label{E-L}\tag{E-L}
	\frac{d}{ds}L_v(\rho(s),\dot{\rho}(s))=L_x(\rho(s),\dot{\rho}(s)),\qquad s\in[a,b],
\end{align}
We call a $C^1$ curve $\rho:[a,b]\to M$ an \emph{extremal} for the Lagrangian $L$ if it satisfies \eqref{E-L}. It is well known that \eqref{E-L} defines a complete Euler-Lagrange flow $\Phi_t^L:TM \to TM$. 

We denote by $d_xA_t (\cdot,y)$ (resp. $ d_yA_t(x,\cdot)$) the differential of $A_t(x,y)$ with respect to the first (resp. second) variable. Similarly, the gradient of $A_t(x,y)$ with respect to the first (resp. second) variable will be denoted by $\nabla_xA_t(\cdot,y)$ (resp. $\nabla_yA_t(x,\cdot)$).


\begin{Pro}\label{pro:differential}
	If $L$ is a Tonelli Lagrangian on the connected closed Riemannian manifold $(M,g)$. Then, the following statements are true:
	\begin{enumerate}[\rm (1)]
		\item $A_t(x, y)$ is differentiable at $y$ if and only if there is unique minimal curve $\rho:[0,t]\to M$ for $A_t(x, y)$. Moreover, if $A_t(x, y)$ is differentiable at $y$, we have
		\begin{align*}
			d_y A_t(x, y)=L_v(\rho(t), \dot{\rho}(t)).
		\end{align*}
		\item $A_t(x, y)$ is differentiable at $x$ if and only if there is unique  minimal curve $\rho:[0,t]\to M$ for $A_t(x, y)$. Moreover, if $A_t(\cdot,y)$ is differentiable at $x$, we have
		\begin{align*}
			d_x A_t(x, y)=-L_v(\rho(0), \dot{\rho}(0)).
		\end{align*}
		\item For any $\tau>0$ there exists a compact subset $K_\tau \subset TM$ satisfies the following property: if $\rho:[0,t]\to M$ is a minimal curve for $A_t(x,y) $ with $t>\tau$, then
		\begin{align*}
			(\rho(s),\dot{\rho}(s))\in K_\tau ,\qquad \forall s\in[0,t].
		\end{align*}
	\end{enumerate}
\end{Pro}


Let $u \in C^0(M)$ and $t>0$. We define respectively the \emph{negative} and \emph{positive Lax-Oleinik operators}: for any $x\in M$, 
\begin{align*}
	T_t^-u(x)=\inf _{y \in M}\{u(y)+A_t(y, x)\},\qquad T_t^+u(x)=\sup _{y \in M}\{u(y)-A_t(x, y)\}.
\end{align*}
As usual we define $T^\pm_0=id$. From weak KAM theory, a function $u$ is a  weak KAM solution of  \eqref{eq:HJ_wk}  if and only if $T^-_tu(x)+c[L]t=u(x)$ for all $t\geqslant0$. This implies that,  if $u:M\to \R $ is a weak KAM solution of \eqref{eq:HJ_wk}, then for any continuous piecewise $C^1$ curve $\rho:[a,b]\to M$, $a<b$,
\begin{align*}
	u(\rho(b))-u(\rho(a)) \leq \int_a^b L(\rho(s) , \dot{\rho}(s)) d s+c[L](b-a).
\end{align*}
A curve $\rho: [a,b]\to M$ is \emph{$(u,L,c[L])$-calibrated} on $[a,b]$, or \emph{$u$-calibrated} for short, if for every $t,s\in[a,b]$ with $t\leqslant s$,
\begin{align*}
	u(\rho(s))-u(\rho(t))=\int_t^s L(\rho(z), \dot{\rho}(z)) d z+c[L](s-t).
\end{align*}
If $u$ is a weak KAM solution of \eqref{eq:HJ_wk}, then for any $x\in M$, there exists a $(u,L,c[L])$-calibrated curve $\rho:(-\infty,0]\to M$ such that $\rho(0)=x$. One can refer to \cite{Cannarsa_Sinestrari_book,Rifford2008} for more in the case when $u$ is not differentiable at $x$.

For the associated Tonelli Hamiltonian $H:T^*M\to\R$ of $L$, in local charts, we have following Hamiltonian ODE
\begin{equation}\label{H-S}
\left\{  
\begin{aligned}
	\dot{x}&=H_p(x,p), \\
	\dot{p}&=-H_x(x,p).
\end{aligned}\right.	
\end{equation}
We call $\Phi_t^H$ the Hamiltonian flow associated with $\Phi_t^L$. The Legendre transform 
\begin{align*}
\mathcal{L}(x,v)=(x,L_v(x,v))
\end{align*}
define a diffeomorphism from $TM$ to $T^*M$, and it establishes a correspondence between the  Euler-Lagrange flow $\Phi_t^L$ and its corresponding Hamilton flow by
\begin{align*}
\Phi_t^H=\mathcal{L}\circ\Phi_t^L\circ\mathcal{L}^{-1}.
\end{align*}

%


\subsection{Facts from Riemannian Geometry }

Let us recall some basic facts about Riemannian geometry. For more details in Riemannian geometry we refer to  \cite{do_Carmo1992book,Petersen_book2016,Sakai1996book}.


Let $\nabla$ be the Riemannian connection on $(M,g)$. The curvature tensor of the Riemannian connection $\nabla$ is defined by
\begin{align*}
R:\Gamma(TM)\times\Gamma(TM)\times\Gamma(TM)\to&\,\Gamma(TM),\\
(X,Y,Z)\mapsto&\,R(X,Y)Z=\nabla_X\nabla_YZ-\nabla_Y\nabla_XZ-\nabla_{[X, Y]}Z.
\end{align*}
The Ricci curvature at $v \in T_p M$ is defined as
\begin{align*}
\operatorname{Ric}(v)=\operatorname{tr}(w\mapsto R(w,v)v).
\end{align*}

Next, we introduce some differential operators on Riemannian manifold. Given a Riemannian manifold $(M, g)$ with its Riemannian connection $\nabla$.
\begin{enumerate}[--]
\item The gradient of a function $f \in C^1(M)$ is given by
\begin{align*}
\nabla: C^1(M) &\,\to C(M),\\
\nabla f\mapsto&\,(df)^{\#}.
\end{align*}
An equivalent definition is that $\nabla f(x)$ is the unique vector in $T_x M$ such that
\begin{align*}
df(v)=g(\nabla f,v)
\end{align*}
for any $v\in T_xM$.
\item The divergence of a vector field $X\in\Gamma(TM)$ is
\begin{align*}
& \operatorname{div}: \Gamma(TM)\to C^{\infty}(M) \\
& \operatorname{div}X\mapsto\operatorname{tr}(Y\mapsto\nabla_YX).
\end{align*}
\item The Laplacian operator is defined as
\begin{align*}
\Delta: C^2(M)\to&\,C(M),\\
\Delta f\mapsto&\,\operatorname{div}\nabla f.
\end{align*}
\item When seen as a $(1,1)$ type tensor, the Hessian of $f\in C^2(M)$ is given by
\begin{align*}
\operatorname{Hess}f:\Gamma(T M)\to&\,\Gamma(TM),\\
\operatorname{Hess}f(X)\mapsto&\,\nabla_X\nabla f.
\end{align*}
\item When viewed as a $(0,2)$ type tensor, the Hessian of $f\in C^2(M)$ is
\begin{align*}
\nabla^2f:\Gamma(TM)\times\Gamma(T M)\to&\,C^{\infty}(M),\\
\nabla^2f(X,Y)\mapsto&\,g(\nabla_X\nabla f,Y).
\end{align*}
\end{enumerate}
In fact, we can consider Hessian operators for a function $f$ that are not of class $C^2$, even not continuously differentiable. For more details about this we refer to \cite{Villani_book2009}.

Given a compact oriented  Riemannian manifold $(M,g)$. Let $\Omega^k(M)$ be the space of $k$-form on $M$ and let $\Omega(M)=\cup_k\Omega^k(M)$. Recall that the Riemannian metric $g$ induces an inner product on $T_p^*M$. Extending this inner product from $T_p^*M$ to its k-th exterior wedge $\bigwedge^k(T_p^*M )$ one can obtain an inner product $\langle\cdot, \cdot\rangle_p$	on $\bigwedge^k(T_p^*M )$. The inner product on $\Omega^k(M)$ is then defined as
\begin{align*}
\Omega^k(M)\times\Omega^k(M)\to&\,\R,\\
(\omega,\eta)\mapsto&\,\int_M\langle\omega,\eta\rangle_p\sigma,
\end{align*}
where $\sigma $ is the volume form associated to $g$. If we require $(\omega,\eta)=0$ for $ \omega\in\Omega^k(M)$, $\eta\in \Omega^l(M)$ with $k\neq l$, we get an inner product $(\cdot,\cdot)$ on $\Omega(M)$.  

Since the exterior differential operator $d :\Omega^k(M)\to\Omega^{k+1}(M)$ is a linear operator on the inner space $(\Omega(M), (\cdot,\cdot))$, one has a linear adjoint operator
\begin{align*}
\delta:\Omega^{k+1}(M)\to\Omega^k(M) 
\end{align*}
of $d$ such that $(d\omega,\eta)=(\omega,\delta\eta)$. The Hodge Laplacian is then defined by 
\begin{align*}
\varDelta=d\delta+\delta d.
\end{align*}
This is a second order linear differential operator. By definition, $\omega\in\Omega(M)$ is a  harmonic form if $\varDelta\omega=0$.

Now we list some basic facts about Hodge cohomology (see, for instance, \cite{Petersen_book2016,Warner1983book}).

\begin{Pro}\label{some facts about Hodge Cohomology}
Let $(M,g)$ be a compact oriented Riemannian manifold.	The following statements hold true:
\begin{enumerate}[\rm (a)]
\item Every harmonic form is closed.
\item  A closed 1-form $\omega$ is harmonic if and only if $\operatorname{div}\omega^\sharp=0$. 
\item (Hodge Theorem): The Hodge Cohomology $\mathcal{H}(M,\R)$ is isomorphic to the De Rham Cohomology $H(M, \R)$. 
\item (Bochner  Theorem): If $(M,g)$ has nonnegative Ricci curvature, then $g(\omega^\sharp,\omega^\sharp)$ is constant for every harmonic 1-form $\omega$. \label{constant harmonic}
\end{enumerate}
\end{Pro}

\subsection{Conjugate Points and Jacobi Fields}
In order to estimate the Laplacian of the generating function $A_t(x, y)$, we need to discuss the  conjugate points and the Jacobi fields. This topic is well known in Riemannian geometry. However, for the sake of convenience we shall deal with these points in the frame of Lagrange geometry (see, also \cite{Contreras_Iturriaga1999} in the Hamiltonian frame). 

\begin{defn}
Suppose $L:TM\to\R$ is a Tonelli Lagrangian on a closed connected Riemannian manifold $(M,g)$ and $\rho:[a,b]\to M$ is an extremal. A variation of extremal curves along $\rho$ is a map $\rho(t,s)\in C^2([a, b]\times(-\varepsilon,\varepsilon))$ satisfying
\begin{enumerate}[\rm (1)]
\item $\rho(t,0)=\rho(t)$ for all $t\in [a,b]$.
\item $\rho(\cdot,s):[a,b]\to M$ is an extremal curve for each $s\in(-\varepsilon, \varepsilon)$. 
\end{enumerate} 
Let $J:[a,b]\to TM$ be a vector field along $\rho$.  We say that $J$ is a \emph{Jacobi field}  if one can find a variation $\rho(t,s)\in C^2([a, b]\times(-\varepsilon,\varepsilon))$ of extremal curves along $\rho$ such that
\begin{align*}
J(t)=\left.\frac{\partial}{\partial s}\right|_{s=0}\rho(t,s).
\end{align*}
\end{defn}

\begin{Pro}\label{pro:Jabobi_field}
Let $L:TM\to\R$ be a  Tonelli Lagrangian on a closed connected Riemannian manifold $(M,g)$. If  $\rho:[a,b]\to M$ is an extremal and $J:[a,b]\to TM$ is a vector field along $\rho$, then $J$ is a Jacobi field along $\rho$ if and only if $J$ solves the second order linear Jacobi equation
\begin{equation}\label{eq:Jacobi}
\frac d{dt}(L_{vx}(\rho(t),\dot{\rho}(t))J(t)+L_{vv}(\rho(t),\dot{\rho}(t))\dot{J}(t)) =L_{xx}(\rho(t),\dot{\rho}(t))J(t)+L_{xv}(\rho(t),\dot{\rho}(t))\dot J(t). 
\end{equation}
in local chart.
\end{Pro}

\begin{Pro}\label{pro:Jacobi_formula}
Suppose $J:[a,b]\to TM$ is a Jacobi field along $\rho:[a,b]\to M$ such that $J(0)=0$. Then 
\begin{align*}
J(s)=\left.\partial_z\right|_{z=0}\pi\circ\Phi_{s-a}^L(\rho(a),\dot{\rho}(a)+z \nabla_{\dot{\rho}(a)}J).
\end{align*}
\end{Pro}

\begin{defn}    	
If $L: TM\to\R$ is a Tonelli Lagrangian and $\rho:[a,b]\to M$ is an extremal. The point $(\rho(b),\dot{\rho}(b))$ is said to be conjugate to $(\rho(a),\dot{\rho}(a))$ if there exists a nonzero Jacobi field $J$ along $\rho$ such that
\begin{align*}
\rho(a)=\rho(b)=0.
\end{align*}
\end{defn}

In general, a Tonelli Lagrangian $ L:TM\to\R$ is not necessarily symmetrical. So one can define the reverse of $L$ by $\breve{L}(x,v)=L(x,-v)$. Simultaneously, one gets the reverse Hamiltonian $\breve{H}(x,p):=H(x,-p)$ which is exactly the Hamiltonian associated to $\breve{L}$. The next proposition clarifies the relation of conjugacy with respect to $L$ and $\breve{L}$ respectively.

\begin{Pro}\label{pro:reverse}
Let $L:TM\to\R$ be a Tonelli Lagrangian and let $\breve{L}(x, v):=L(x,-v)$ for all $(x,v)\in TM$. If $\rho:[a,b]\to M $ is an extremal, then, $(\rho(b),\dot{\rho}(b))$ is conjugate to $(\rho(a),\dot{\rho}(a))$ with respect to $L$ if and only if $(\rho(a),-\dot{\rho}(a))$ is conjugate to $(\rho(b),-\dot{\rho}(b))$ with respect to $\breve {L}$.
\end{Pro}

\begin{Pro}\label{pro:2nd_var}
Given a mechanical Lagrangian $L(x, v)=\frac 12g_x(v,v)-f(x)$ on a connected closed manifold $(M,g)$. Let $\omega$ be a closed 1-form on $M$ and let $X:=\omega^{\#}$ be its corresponding vector field. Then, any minimizer $\rho:[0,t]\to M$ of $A_t(\rho(0),\rho(t))$ associated to the Lagrangian
\begin{align*}
L_\omega(x,v):=L(x,v)-\omega(v)=L(x,v)-g_x(X, v)
\end{align*}
solves
\begin{equation}\label{eq:second_law}
\nabla_{\dot{\rho}}\dot{\rho}=-\nabla f,
\end{equation}
and each Jacobi field $J$ along $\rho$ satisfies
\begin{equation}\label{Jacobi equation}
\nabla_{\dot{\rho}}\nabla_{\dot{\rho}}J+R(J,\dot{\rho})\dot{\rho}+\operatorname{Hess} f(J)=0.
\end{equation}
\end{Pro}

Equation \eqref{eq:second_law} is equivalent to the Euler-Lagrange equation \eqref{E-L}. Indeed, we have $L(x,v)=\frac{1}{2} g_{i j}(x)v^i v^j-f(x)-\omega_lv^l$ in local chart. Here and after we use the Einstein summation convention. Then, we have
\begin{align*}
&\,L_{v_k}=g_{k i}(x)v^i-\omega_k,\\
&\,L_{x_k}=\frac 12\frac{\partial g_{ij}}{\partial x_k}(x)v^iv^j-\frac{\partial f}{\partial x_k}(x)-\frac{\partial\omega_l}{\partial x_k}(x)v^l.
\end{align*}
The Euler-Lagrange equation \eqref{E-L} tells us
\begin{align*}
\frac{d}{d t}L_{v_k}(\rho(t),\dot{\rho}(t))=L_{x_k}(\rho(t),\dot{\rho}(t))
\end{align*}
which yields to
\begin{align*}
\frac 12\frac{\partial g_{ij}}{\partial x_k}(\rho)\dot{\rho}^i\dot{\rho}^j-\frac{\partial f}{\partial x_k}(\rho)-\frac{\partial\omega_l}{\partial x_k}(\rho)\dot{\rho}^l
=&\,\frac{\partial g_{ki}}{\partial x_j}(\rho)\dot{\rho}^i\dot{\rho}^j+g_{ki}(\rho)\ddot{\rho}^i-\frac{\partial\omega_k}{\partial x_j}(\rho)\dot{\rho}^j\\
=&\,\frac 12\frac{\partial g_{ki}}{\partial x_j}(\rho)\dot{\rho}^i\dot{\rho}^j+\frac 12\frac{\partial g_{kj}}{\partial x_i}(\rho)\dot{\rho}^i\dot{\rho}^j+g_{ki}(\rho)\ddot{\rho}^i-\frac{\partial\omega_k}{\partial x_j}(\rho) \dot{\rho}^j.
\end{align*}
Since $ \omega $ is closed, we have
\begin{align*}
\frac{\partial\omega_l}{\partial x_k}(\rho)\dot{\rho}^l=\frac{\partial\omega_j}{\partial x_k}(\rho)\dot{\rho}^j=\frac{\partial \omega_k}{\partial x_j}(\rho)\dot{\rho}^j.
\end{align*} 
This means that
\begin{align*}
\frac 12\frac{\partial g_{i j}}{\partial x_k}(\rho)\dot{\rho}^i\dot{\rho}^j-\frac{\partial f}{\partial x_k}(\rho)=\frac 12\frac{\partial g_{ki}}{\partial x_j}(\rho) \dot{\rho}^i\dot{\rho}^j+\frac 12\frac{\partial g_{kj}}{\partial x_i}(\rho)\dot{\rho}^i \dot{\rho}^j+g_{ki}(\rho)\ddot{\rho}^i.
\end{align*}
Hence, we obtain
\begin{equation}\label{coordinate equation} 
\begin{split}
\ddot{\rho}^k=&\,-g^{km}\frac{1}{2}\left\{\frac{\partial g_{mi}}{\partial x_j}\dot{\rho}^i \dot{\rho}^j+\frac{\partial g_{m j}}{\partial x_i}  \dot{\rho}^i\dot{\rho}^j-\frac{\partial g_{ij}}{\partial x_m}\dot{\rho}^i\dot{\rho}^j\right\}-g^{km}\frac{\partial f}{\partial x_m}\\
=&\,-\Gamma_{ij}^k\dot{\rho}^i\dot{\rho}^j-g^{km}\frac{\partial f}{\partial x_m}, 
\end{split}
\end{equation}
where
\begin{align*}
\Gamma_{ij}^k=g^{km}\frac 12\left\{\frac{\partial g_{mj}}{\partial x_i}+\frac{\partial  g_{mi}}{\partial x_j}-\frac{\partial g_{ij}}{\partial x_m}\right\}
\end{align*}
are Christoffel symbols. Notice that equation \eqref{eq:second_law} is also equivalent to equation \eqref{coordinate equation} in local chart. Therefore, equation \eqref{eq:second_law} is equivalent to the Euler-Lagrange equation \eqref{E-L}.

\begin{Lem}\label{lem:not_conjugate}
Suppose $L:TM\to\R$ is a Tonelli Lagrangian on a closed connected Riemannian manifold $(M,g)$. Let $\rho:[a,b]\to M$ be an extremal. Then $(\rho(b),\dot{\rho}(b))$ is not conjugate to $(\rho(a),\dot{\rho}(a))$, $s>t$ if and only if $d_{\dot{\rho}(a)}(\pi\circ\Phi_{b-a}^L)$ is non-degenerate.
\end{Lem}

Now we come to the connection between the conjugate points and the differentiability of $A_t(x,\cdot)$.

\begin{Pro}\label{pro:before_conjugate}
Suppose $\rho:[0,t]\to M$ is a minimal curve for $A_t(x,y)$ and $(\rho(t),\rho(t))$ is not conjugate to $(\rho(0),\dot{\rho}(0))$. Then $A_t(x,\cdot)$ is of class $C^r$ in a neighborhood $U$ of $\rho(t)$ provided $A_t(x,\cdot)$ is differentiable at $\rho(t)$. Moreover,
\begin{align*}
\pi\circ\Phi_t^L:(\pi\circ\Phi_t^L)^{-1}(U)\to U
\end{align*}
is a $C^{r-1}$ diffeomorphism and the curve $\rho_z(s):=\pi\circ\Phi_s^L(x,v_z)$ is the unique minimal curve  for $A_t(x, z)$ where $T_xM\ni v_z=(\pi\circ\Phi_t^L)^{-1}(z)$.
\end{Pro}

Now we introduce the definition of the cut points which play an important role in calculus of variations.

\begin{defn}
Suppose $\rho:[0,t]\to M$ is a minimal curve for $A_t(x,y)$. Then, $\rho(t)$ is a cut point of $\rho(0)$ if the curve $\bar{\rho}(s)=\pi\circ\Phi_s^L(x,\dot{\rho}(0))$, $s\in[0,\tau]$, is not a minimal curve for $A_\tau(x, \bar{\rho}(\tau))$ for any $\tau>t$.
\end{defn}

\begin{Lem}\label{lem:cut}
If $\rho:[0,t]\to M$ is a minimal curve for $A_t(x,y)$ and $\rho(t)$ is a cut point of $\rho(0)$, then either
\begin{enumerate}[\rm (i)]
\item $(\rho(t),\dot{\rho}(t))$ is conjugate to $(\rho(0),\dot{\rho}(0))$, or
\item there exists another minimizer $\tilde{\rho}:[0, t] \rightarrow M$ of $A_t(x, y)$.
\end{enumerate}
\end{Lem}

\begin{The}\label{thm:conjugate_weak_KAM}
Let $L:TM\to\R$ be a Tonelli Lagrangian on a closed connected Riemannian manifold $(M, g)$ and let $u$ be a weak KAM solution of \eqref{eq:HJ_wk}. Suppose that $x\in M$ and $\rho:(-\infty,0]\to M$ is a $(u,L,c[L])$-calibrated curve ending at $\rho(0)=x$.
\begin{enumerate}[\rm (1)]
\item $(\rho(0),\dot{\rho}(0))$ is not conjugate to $(\rho(-s),\dot{\rho}(-s))$ for any $s>0$.
\item For every $\tau>0$, $A_\tau(\rho(-\tau),\cdot)$ is of $C^2$ in a neighborhood of $\rho(0)=x$ and $\rho(0)$ is not a cut point of $\rho(-\tau)$.
\end{enumerate}
\end{The}

\begin{proof}
Statement (1) is a direct consequence of Proposition \ref{pro:conjugate}. To prove (2), it is sufficient to prove $x$ is a point of differentiability of $A_\tau(\rho(-\tau), \cdot)$, for any $\tau>0$, by Proposition \ref{pro:before_conjugate} and Lemma \ref{lem:cut}. Otherwise, there exists another minimizer $\alpha:[0, \tau]\to M$ of $A_\tau(\rho(-\tau),x)$. Then, the speed curve of
\begin{align*}
\alpha_1(t)=
\begin{cases}
	\rho(t), & t\in[-\tau-1,-\tau]\\
	\alpha(t+\tau), & t\in[-\tau, 0]
\end{cases}
\end{align*}
satisfies \eqref{E-L} which contradicts to the Cauchy-Lipschitz Theorem. 
\end{proof}

\section{Riccati Equation }

Now we turn to the associated Riccati equation which will help us to estimate the Laplacian of the fundamental solution $A_t(x,y)$ and the Laplacian of weak KAM solution in the barrier sense. Our approach is inspired by the Gromoll and Cheeger's  splitting theorem on a non-compact manifold with nonnegative Ricci curvature. The properties of rays  ensure the non-existence of conjugate point (or cut point) in positive direction. Fortunately, in the compact case, Fathi's weak KAM theorem constructs the backward calibrated curves which play the same role as the rays. In principle, the key point in the following discussion is to avoid the trouble of the existence of conjugate points.

In this section, we  suppose that  $\dim M =n\geqslant2$  and the Lagrangian $L$ has the form
\begin{align*}
L(x, v)=\frac 12g_x(v,v)-f(x)-\omega(v)=\frac 12g_x(v,v)-f(x)-g_x(X,v)
\end{align*}
where $\omega$ is a closed 1-form on $M$ and  $X:=\omega^{\#}$ be its corresponding vector field.

\begin{The}\label{thm:Riccati_1}
If $\rho:[0,t]\to M$ is a minimal curve for $A_t(\rho(0),\rho(t))$ and $(\rho(t),\dot{\rho}(t))$ is not conjugate to $(\rho(0),\dot{\rho}(0))$, then we have
\begin{equation}\label{eq:Riccati_1}
\dot{\Theta}(s)+\frac{1}{n}\Theta^2(s)+\operatorname{Ric}(\dot{\rho}(s))+\Delta f(\rho(s))\leqslant0,\quad s\in(0,t],
\end{equation}
where $\Theta(s)=\Delta_y A_s(\rho(0),\rho(s))+\operatorname{div}X(\rho(s)).$ 
\end{The}

\begin{proof}
Let $(e_1, e_2, \cdots, e_n)$ be an orthonormal basis of $T_{\rho(0)} M$ and let us parallel transport along $\rho$ to define a new family $(e_1(s), e_2(s), \ldots, e_n(s))$ in $T_{\rho(s)} M$. Denote $\pi \circ \Phi_s^L(x, \dot{\rho}(0)+z e_i)$ by $\rho_i(s, z), i=1,2, \cdots, n$. By Proposition \ref{pro:Jacobi_formula},
\begin{align*}
J_i(s)=\left.\partial_z\right|_{z=0}\rho_i(s,z),\qquad i=1,2,\ldots,n,
\end{align*}
are Jacobi fields with $J_i(0)=0, \nabla_{\dot{ \rho}} J_i(0)=e_i$. Let $J_i(s)=\sum_{j=1}^n a_{i j}(s) e_j(s), i=1,2, \cdots, n$, then
\begin{align*}
\sum_{j=1}^n\ddot{a}_{ij}(s) e_j(s)+\sum_{j=1}^n a_{ij}(s)R(e_j(s),\dot{\rho}(s))\dot{\rho}(s)+\sum_{j=1}^n a_{ij}(s)\operatorname{Hess}f(e_j(s))=0
\end{align*}
by \eqref{Jacobi equation}. This implies that
\begin{align*}
\ddot{a}_{ik}(s)+\sum_{j=1}^na_{ij}(s)g(R(e_j(s),\dot{\rho}(s))\dot{\rho}(s),e_k(s))+\sum_{j=1}^na_{ij}(s)g(\operatorname{Hess}f(e_j(s)),e_k(s))=0
\end{align*}
for $i,k=1,2,\ldots,n$. Therefore, we have the following matrix Riccati equation
\begin{align*}
\ddot{A}(s)+A(s)R(s)+A(s)\nabla^2f(s)=0,
\end{align*}
where
\begin{align*}
A(s)=&\,(a_{ij}(s))_{1\leqslant i,j\leqslant n},\\
R(s)=&\,(g(R(e_i(s),\dot{\rho}(s))\dot{\rho}(s),e_j(s))_{1\leqslant i,j\leqslant n},\\
\nabla^2f(s)=&\,g(\operatorname{Hess}f(e_i(s)),e_j(s))_{1 \leqslant i,j\leqslant n}.
\end{align*}
We now claim that
\begin{align*}
\nabla_{\dot{\rho_i}}J_i(s)=a_{ij}(s)\operatorname{Hess}A_s^x(e_j(s)),
\end{align*}
where $A_s^x(\cdot):=A_s(\rho(0),\cdot)$.

Indeed, by Proposition \ref{pro:before_conjugate} we have
\begin{align*}
\nabla_{\dot{\rho}}J_i(s)=&\,\nabla_{\partial_s\rho_i}\partial_z\rho_i(s, 0)=\nabla_{\partial_z\rho_i}\partial_s\rho_i(s,0)=\nabla_{\partial_z\rho_i}\frac{\partial}{\partial s}\pi\circ\Phi_s^L(x,\dot{\rho}(0)+z e_i)(s,0)\\
=&\,\nabla_{\partial_z \rho_i}\Phi_s^L(x, \dot{\rho}(0)+z e_i)(s,0)=\nabla_{\partial_z\rho_i}(\rho_i(s,z),\dot\rho_i(s,z))(s,0),
\end{align*}
where for any $r\in[0,s]$, $\rho_i(r,z)=\pi\circ\Phi_r^L(\rho(0),\dot{\rho}(0)+ze_i)$ is the unique minimal curve for $A_s(\rho(0),\pi\circ\Phi_s^L(x,\dot{\rho}(0)+ze_i))$ for $z$ small enough.

By Proposition \ref{pro:differential} we have $L_v(\rho_i(s,z),\dot{\rho}_i(s,z))=d_yA_s(\rho(0),\rho_i(s,z))$, and this implies that
\begin{align*}
(\rho_i(s,z),\dot\rho_i(s,z))=\mathcal{L}^{-1}(\rho_i(s,z),d_yA_s(\rho(0),\rho_i(s,z)))=\nabla_yA_s(\rho(0),\rho_i(s,z))+X(\rho_i(s,z)), 
\end{align*}
where $H(x,p)=\frac 12g_x^*(p+\omega,p+\omega)+f(x)$ is the  Tonelli Hamiltonian associated to $L$. Therefore, 
\begin{align*}
\nabla_{\dot{\rho}}J_i(s)=&\,\nabla_{\partial_z\rho_i}(\rho_i(s,z),\dot{\rho}_i(s,z))(s,0)=\nabla_{J_i(s)}(\nabla_y A_s^x+X)\\
=&\,\nabla_{\sum_{j=1}^na_{ij}(s)e_j(s)}(\nabla_yA_s^x+X)=\sum_{j=1}^na_{ij}(s)(\operatorname{Hess}A_s^x(e_j(s))+\nabla_{e_j(s)}X).
\end{align*}

Notice that
\begin{align*}
g(\nabla_{\dot{\rho}}J_i(s),e_k(s))=&\,\sum_{j=1}^n a_{ij}(s)\{g(\operatorname{Hess}A_s^x(e_j(s),e_k(s))+g(\nabla_{e_j(s)}X,e_k(s))\} \\
=&\sum_{j=1}^n \dot{a}_{i j}(s) g(e_j(s), e_k(s)).
\end{align*}
Thus, we obtain
\begin{align*}
A(s)(\nabla^2A_s^x+B(s))=\dot{A}(s),
\end{align*}
where $\nabla^2 A_s^x=(g(\operatorname{Hess}A_s^x(e_i(s)),e_j(s))_{1\leqslant i,j\leqslant n}$ and $B(s)=(g(\nabla_{e_i(s)}X, e_j(s)))_{1\leqslant i, j\leqslant n}$.

Since $(\rho(s),\dot{\rho}(s))$ is not conjugate to $(\rho(0),\dot{\rho}(0))$, $J_i(s)=d_{\dot{\rho}(0)}(\pi\circ\Phi_s^L)(e_i)$ are linear independent and $A(s)$ is invertible for each $s\in(0,t]$. Let $\Lambda(s)=A^{-1}(s) \dot{A}(s)=\nabla^2 A_s^x+B(s)$. We get that
\begin{align*}
\dot{\Lambda}(s)=&\,-A^{-1}(s)\dot{A}(s)A^{-1}(s)\dot{A}(s)+A^{-1}(s)\ddot{A}(s)\\
=&\,-\Lambda^2(s)+A^{-1}(s)(-A(s)R(s)-A(s)\nabla^2f(s))\\
=&\,-\Lambda^2(s)-R(s)-\nabla^2 f(s)
\end{align*}
and $\operatorname{tr}\Lambda(s)=\Delta_yA_s(\rho(0),\rho(s))+\operatorname{div}X(\rho(s))$. We rewrite the equality above as
\begin{equation}\label{eq:Riccati_equation}
\dot{\Lambda}(s)+\Lambda^2(s)+R(s)+\nabla^2f(s)=0.
\end{equation}
By taking trace of \eqref{eq:Riccati_equation}, we arrive at
\begin{align*}
\frac{d}{dt}(\operatorname{tr}\Lambda(s))+\operatorname{tr}(\Lambda^2(s))+\operatorname{Ric}(\dot{\rho}(s))+\Delta f(\rho(s))=0.
\end{align*}
Set $\operatorname{tr}\Lambda(s)=\Theta(s)$. Then \eqref{eq:Riccati_1} follows by recalling the inequality $\operatorname{tr}(\Lambda^2(s))\geqslant\frac 1n\operatorname{tr}^2(\Lambda(s))$.		
\end{proof}

Next, we give some basic comparison estimates of the Riccati equation that will be needed later.
\begin{Lem}\label{lem:Riccati_comparison}
Consider a $C^1$ function $\alpha:(0,t)\to\R$ such that
\begin{align*}
\dot{\alpha}(s)+\frac 1n\alpha^2(s)+k\leqslant0,\qquad \lim _{s\to0^+}s^2\alpha(s)=0.
\end{align*}
Then, 
\begin{align*}
\alpha(s)\leqslant
\begin{cases}
	\sqrt{nk}\cot(\sqrt{k/n}s)&\text{if}\ k>0, s<\min\{t,\pi/\sqrt{k/n}\}, \\
	n/s & \text{if}\ k=0, \\
	\sqrt{-nk}\operatorname{coth}(\sqrt{-k/n}s)&\text{if}\ k<0 .
\end{cases}
\end{align*}
for any $s\in(0,t)$.
\end{Lem}

\begin{proof}
Set
\begin{align*}
S_{n, k}(s):=
\begin{cases}
	\sqrt{n/k}\sin(\sqrt{k/n}s)&\text{if}\ k>0,\\
	s&\text{if}\ k=0, \\
	\sqrt{-n/k}\operatorname{sinh}(\sqrt{-k/n}s)&\text{if}\ k<0 .
\end{cases}
\end{align*}
then, $\beta(t):=n\dot{S}_{n,k}(s)/S_{n,k}(s)$ solves the Riccati equation
\begin{align*}
\dot{\alpha}(s)+\frac{1}{n}\alpha^2(s)+k=0,\quad s \in(0,t).
\end{align*}
Inspired by (3.10) in \cite{Wei_Wylie2009}, we have
\begin{align*}
\frac{d}{dt}(S_{n,k}^2(\alpha-\beta))=&\,2S_{n,k}\dot{S}_{n,k}(\alpha-\beta)+S_{n,k}^2(\dot{\alpha}-\dot{\beta})\\
\leqslant&\,2S_{n,k}\dot{S}_{n,k}(\alpha-\beta)+S_{n,k}^2(-\frac 1n\alpha^2-k+\frac 1n\beta^2+k)\\
=&\,2S_{n,k}\dot{S}_{n,k}(\alpha-\beta)+\frac 1nS_{n,k}^2(\beta^2-\alpha^2)\\
=&\,\frac{2}{n}S_{n,k}^2\beta(\alpha-\beta)+\frac 1nS_{n,k}^2(\beta^2-\alpha^2)\\
=&\,-\frac{1}{n}S_{n, k}^2(\beta-\alpha)^2\leqslant0.
\end{align*}
Together with the condition that $\lim_{s\to0^+}S_{n,k}^2\alpha=\lim_{s\to0^+}S_{n,k}^2 \beta=0$, we have $\alpha(s)\leqslant\beta(s)$ for $s\in(0, t)$.
\end{proof}

\section{Main Results}
In this section, we show certain rigidity results for the weak KAM solutions and Aubry sets under certain curvature hypothesis. Moreover, we give some applications to the Mather quotient and   Ma\~n\'e's  Lagrangian. 

Now, we recall the notion of Laplacian of a continuous function in the barrier sense.
\begin{defn}
Let $f:M\to\R$ be a continuous function on a Riemannian manifold $(M,g)$.
\begin{enumerate}[\rm (1)]
\item A $C^2$ function $\hat{f}:M\to\R$ is said to be a support function from above of $f$ at $p\in M$ if $\hat{f}(p)=f(p)$ and $\hat{f}(x)\geqslant f(x)$ in some neighborhood of $p$. 
\item We say $\Delta f(p)\leqslant B\in\R$ in the barrier sense if for every $\epsilon>0$,  one can find a $C^2$ support function $f_{\epsilon}$ from above of $f$ at $p$ such that $\Delta f_{\epsilon}(p)\leqslant B+\epsilon$.
\item A continuous function $f:M\to\R$ is said to be superharmonic if $\Delta f(p)\leqslant0$ in the barrier sense for each $p\in M$.
\item Similarly, we say that a continuous function $f:M\to\R$ is subharmonic if $-f$ is superharmonic. 
\end{enumerate} 
\end{defn}

The following maximal principle was proved by Calabi in \cite{Calabi1958}. A fundamental proof can be aslo found in \cite{Eschenburg_Heintze1984}.

\begin{The}\label{thm:MP}
If $f:M\to\R$ is a superharmonic function, then $f$ is constant in a neighborhood of every local minimum. In particular, $f$ is constant if $f$ has a global minimum.
\end{The}

\begin{Lem}\label{lem:barrier}
Let $L:TM\to\R$ be a Tonelli Lagrangian on a closed connected Riemannian manifold $(M,g)$ and let $u$ be a weak KAM solution of \eqref{eq:HJ_wk}. For any point $x$ and any $(u,L,c[L])$-calibrated curve $\rho:(-\infty,0]\to M$ ending at $x$, the function
\begin{align*}
u(\rho(-t))+ A_t(\rho(-t),\cdot)+c[L]t
\end{align*}
defined on $M$ is a support function from above of $u$ at $x$ for any $t>0$.  
\end{Lem}

\begin{proof}
Since $\rho$ is a  $(u,L,c[L])$-calibrated curve, we have
\begin{align*}
u(x)=T_t^-u(x)+c[L]t=u(\rho(-t))+A_t(\rho(-t),x)+c[L]t
\end{align*}
for any $t>0$. In addition,
\begin{align*}
u(y)=T_t^{-}u(y)+c[L]t \leqslant u(\rho(-t))+A_t(\rho(-t),y)+c[L]t,\qquad\forall y\in M.
\end{align*}
By Proposition \ref{pro:conjugate}, $(x.\dot{\rho}(0))$ is not conjugate to $(\rho(-t),\dot{\rho}(-t))$ and Theorem \ref{thm:conjugate_weak_KAM} ensures that the function $\hat{u}(\cdot)=u(\rho(-t))+A_t(\rho(-t),\cdot)+c[L]t$ is of $C^2$ in a neighborhood of $x$. Thus, $\hat{u}$ is a support function from above of $u$ at $x$.
\end{proof}

\begin{Lem}\label{lem:estimate_near_0}
Assume that $L:TM\to\R$ is a Tonelli Lagrangian on a closed connected Riemannian manifold $(M, g).$ If $\rho:[0,t]\to M$ is a minimal curve for $A_t(\rho(0),\rho(t))$ satisfying $|\dot\rho|<K$ for some constant $K>0$, then there exist $\tau>0$ and $C_1,C_2 >0$ such that
\begin{align*}
|\Delta_yA_s(\rho(0),\rho(s))|\leqslant\frac{C_1}{s}+C_2
\end{align*}
for all $s\in(0,\tau)$, with constants $\tau$, $C_1$ and $C_2$ depend only on $K$.
\end{Lem}

\begin{proof}
Choose a coordinate chart $(U,\phi)$ of $\rho(0)$, then one can find $\tau^1>0$ so that 
$\rho(s)\in U$ for all $s\in (0,\tau^1)$. This allows us to reduce to the case when $M=\R^n$.

Note that $L_v(\rho,\dot{\rho})$ is bounded. Applying Lemma 2.4 in \cite{Bernard2012} one can find $\hat{\tau}=\hat{\tau}(K)\in(0,\tau^1)$ and $\hat{C}(K)>0$ such that
\begin{align*}
|d^2_yA_s(\rho(0),\rho(s))|\leqslant\frac{\hat{C}_K}{s},\qquad\forall s\in (0,\hat{\tau})
\end{align*}  
Recall that in local chart the Laplacian operator has the representation
\begin{align*}
\Delta u=\frac{1}{\sqrt{ \det g}}\frac{\partial}{\partial x^i}\big(g^{ij}\sqrt{\det g}  \frac{\partial u}{\partial x^j}\big).
\end{align*}
Together with the boundedness of $d_y A_s(\rho(0),\rho(s))=L_v(\rho(s),\dot{\rho}(s))$, there are  $\tau\in(0,\hat \tau)$ and $C_1\geqslant\hat{C}_K, C_2>0$ such that
\begin{align*}
|\Delta_yA_s(\rho(0),\rho(s))|\leqslant\frac{C_1}{s}+C_2,\qquad\forall s\in (0,\tau),
\end{align*} 
with constants $\tau ,C_1, C_2$ depending only on $K$.
\end{proof}

\begin{proof}[ Proof of Theorem \ref{mechnaical case}]
For all $x\in M $ and each $(u,L,c[L])$-calibrated curve $\rho:(-\infty,0] \to M$ ending at $x,$
we  claim that  $\Delta_yA_t(\rho(-t),x)\leqslant\sqrt{-nk}\operatorname{coth}(\sqrt{-k/n}t)$ for every $t>0$.

Notice that
\begin{align*}
\tilde{\rho}(s)=\rho(s-t),\qquad s\in[0, t]
\end{align*}
is a minimal curve for $A_t(\tilde{\rho}(0),\tilde{\rho}(t))$. By Theorem \ref{thm:conjugate_weak_KAM}, $(\tilde{\rho}(t),\dot{\tilde{\rho}}(t))$ is not conjugate to $(\tilde{\rho}(0),\dot{\tilde{\rho}}(0))$. Applying Theorem \ref{thm:Riccati_1} we obtain
\begin{align*}
\dot{\Theta}(s)+\frac{1}{n}\Theta^2(s)+\operatorname{Ric}(\dot{\tilde{\rho}}(s))+\Delta f(\tilde{\rho}(s))\leqslant0,\qquad s \in(0, t],
\end{align*}
where $\Theta(s)=\Delta_y A_s(\tilde{\rho}(0), \tilde{\rho}(s))$. Since $M$ has nonnegative Ricci curvature and $\Delta f \geqslant k$, we arrive at
\begin{align*}
\dot{\Theta}(s)+\frac{1}{n}\Theta^2(s)+k\leqslant\dot{\Theta}(s)+\frac{1}{n}\Theta^2(s)+\Delta f(\tilde{\rho}(s))\leqslant0,\qquad s \in(0,t].
\end{align*}
Using that $\lim_{s\to0^+}s^2\Theta=0$ (see Lemma \ref{lem:estimate_near_0}) together with Lemma \ref{lem:Riccati_comparison} we discover 
\begin{align*}
\Theta(t)\leqslant\sqrt{-nk}\operatorname{coth}(\sqrt{-k/n}t),\qquad t\in(0,+\infty).
\end{align*}

From the claim we deduce that
\begin{align*}
\Delta u(x) \leqslant \sqrt{-n k}
\end{align*}
in the barrier sense.
\end{proof}

\begin{proof}[Proof of Theorem \ref{solution and divergence} ]
For each point $x\in M$, there exists a  $(u,L,c[L])$-calibrated curve $\rho:(-\infty,0] \to M$ such that
\begin{align*}
\rho(0)=x,\ H(x,L_v(\rho(0),\dot{\rho}(0))=c[L],\ \rho(s)=\pi\circ\Phi_s^L(\rho(0),\dot{\rho}(0)).
\end{align*}
Then, for each $t>0$, the curve
\begin{align*}
\tilde{\rho}(s)=\rho(s-t),\qquad s \in[0,t],
\end{align*}
is a minimal curve  for $A_t(\tilde{\rho}(0),\tilde{\rho}(t))$ and $(\tilde{\rho}(t), \dot{\tilde{\rho}}(t))$ is not conjugate to $(\tilde{\rho}(0), \dot{\tilde{\rho}}(0))$. By Theorem \ref{thm:Riccati_1} we have that
\begin{align*}
\dot{\Theta}(s)+\frac{1}{n}\Theta^2(s)+\operatorname{Ric}(\dot{\tilde{\rho}}(s))+\Delta f(\tilde{\rho}(s))\leqslant0,\qquad s \in(0, t],
\end{align*}
where $\Theta(s)=\Delta_y A_s(\tilde{\rho}(0), \tilde{\rho}(s))+   \operatorname{div}\omega^{\#} (\tilde{\rho}(s))$.

Invoking Lemma \ref{lem:Riccati_comparison} and the fact $\lim _{s \rightarrow 0^{+}} s^2 \Theta =0$ (see Lemma \ref{lem:estimate_near_0}) we obtain
\begin{align*}
\Theta(t)=\Delta_y A_t(\tilde{\rho}(0),\tilde{\rho}(t)) +   \operatorname{div}\omega^{\#} (x)=\Delta_yA_t(\rho(-t), x) +   \operatorname{div}\omega^{\#} (x) \leqslant n/t,\qquad \forall  t>0 .
\end{align*}
Hence, $ \Delta u (x )\leqslant -   \operatorname{div}\omega^{\#} (x) $ in the barrier sense.  
\end{proof}

\begin{proof}[Proof of Theorem   \ref{constant if and only if }]
We first prove the  sufficiency.  If $\omega$ is a harmonic 1-form, then $   \operatorname{div}\omega^{\#}\equiv0$ by Proposition \ref{some facts about Hodge Cohomology} (2). Since $(M,g)$ has nonnegative Ricci curvature,  by Theorem \ref{solution and divergence} each weak KAM solution $u$ of \eqref{eq:HJ_wk} is superharmonic. Hence,   Theorem \ref{thm:MP} implies that $u$ must be constant.

Now we turn to prove the necessity.	If each weak KAM solution $u$ of \eqref{eq:HJ_wk} is constant, Theorem \ref{solution and divergence} shows that $0\leqslant-\operatorname{div}\omega^{\#}(x)$ for all $x\in M$. By Stoke's Theorem we can obtain that $   \operatorname{div}\omega^{\#}\equiv0$. Hence, $ \omega$ is a harmonic 1-form by Proposition \ref{some facts about Hodge Cohomology} (2).
%
%
\end{proof}

\begin{Rem}
If $(M,g)$ is the flat tours $\mathbb{T}^n$ with dimension $n\geqslant 2$, then the de Rham Cohomology  $H^1(\mathbb{T}^n,\R)=\mathbb{R}^n$. The authors in  \cite{Contreras_Iturriagabook1999} proved that each weak KAM solution associated to the Lagrangian
\begin{align*}
L(x,v)=\frac 12\langle v,v\rangle^2-\omega(v),\quad [\omega]\in H^1(\mathbb{T}^n,\R)=\R^n,
\end{align*}
must be constant (see section 5.5 in \cite{Contreras_Iturriagabook1999}).
\end{Rem}
%
%

\begin{proof}[Proof of Corollary \ref{mather quotient}]
By Proposition \ref{some facts about Hodge Cohomology} (3), the inclusion 
\begin{align*}
i:\mathcal{H}^1(M,\R)&\to H^1(M,\R) \\
\omega &\mapsto[\omega ]
\end{align*}
is an isomorphism. For each $[\omega ]\in H^1(M,\R)$, one can choose a representative element $\tilde \omega$ which is a harmonic 1-form. Thus, we can reduce to the case when $\omega \in\mathcal{H}^1(M,\R)$.

Now, we claim  that $h(x,y)\equiv0$ for any $(x,y)\in M\times M$. Since
\begin{align*}
u(y)-u(x)\leqslant h(x,y),\qquad\forall (x,y)\in M\times M
\end{align*}
for any weak KAM solution $u$ of \eqref{eq:HJ_wk}. We observe that $ h(x,y)\geqslant0$ since  $u$ must be constant by Theorem \ref{constant if and only if }.  Now taking a point $z\in   \mathcal{A}(L_\omega)$ we have $h(x,y)\leqslant h(x,z)+h(z,y)$ for all $x,y\in M$. Notice that $h_z(\cdot):=h(z,\cdot)$ is a weak KAM solution of \eqref{eq:HJ_wk} and $h(z,z)=0$. It yields that $ h_z(\cdot)=h(z,\cdot)\equiv0$ by Theorem \ref{constant if and only if }.  In addition, because $ h^z(\cdot):=h(\cdot,z)$ is a weak KAM solution of  \eqref{eq:HJ_wk}  associated to the Lagrangian
$$\breve{L}(x,v):= L(x,-v) =\frac{1}{2} g_x(v, v)+\omega(v) $$
and $ - \omega  $ is also a harmonic 1-form, we know from Theorem \ref{constant if and only if } that $h^z(\cdot)\equiv0.$
Therefore, 
\begin{align*}
0\leqslant h(x,y)\leqslant h(x,z)+h(z,y)=0, \quad \forall (x,y)\in M\times M.
\end{align*}
The fact that $h(x,y)=0$ for every $(x,y)\in M\times M$ implies that $h(x,x)\equiv0$ for all $x\in M$ and $\delta(x,y)=h(x,y)+h(y,x)\equiv0$ for any $(x,y) \in M\times M.$ This completes the proof.
\end{proof}

Now we recall the following result obtained in \cite[Corollary 4.1.13]{Fathi_Figalli_Rifford2009}.
\begin{Pro}\label{differ by a constant }
Let $L:TM\to\R$ be a Tonelli Lagrangian on a closed connected Riemannian manifold $(M,g).$ 
Then,   the Mather quotient $(\mathcal{A}(L),\sim, \delta)$ is a singleton if and only if any two  weak KAM solutions of  \eqref{eq:HJ_wk} differ by a constant.
\end{Pro}

\begin{proof}[Proof of Theorem \ref{mane constant}]
Indeed, the first part of the proof proof has the same reasoning as that of Theorem \ref{constant if and only if }.

If $\omega$ is a harmonic 1-form, then $\operatorname{div}\omega^{\#}(x)=0$ for all $x\in M$.  Invoking Theorem \ref{solution and divergence}, each weak KAM solution  $u$  of \eqref{eq:HJ_wk} is superharmonic. Therefore, $u$ must be constant.

Conversely, if each weak KAM solution $u$ of \eqref{eq:HJ_wk} is constant, Theorem \ref{solution and divergence} implies that $0\leqslant-\operatorname{div}\omega^{\#}(x)    $ for all $x\in M$. Appying Stoke's Theorem one can derive that $   \operatorname{div}\omega^{\#}\equiv0$. Thus $\omega$ is a harmonic 1-form.

Notice that the constant $0$ is a weak KAM solution of \eqref{eq:HJ_wk}. Together with Proposition \ref{differ by a constant }, we find that the Mather quotient $(\mathcal{A}(L),\sim, \delta)$ is a singleton if and only if $ \omega$ is a harmonic 1-form.

Finally, we turn to the rest of the proof. Since $(M,g)$ has nonnegative Ricci curvature and $ \omega$ is a harmonic 1-form, by Proposition \ref{some facts about Hodge Cohomology} (d) we obtain that $f(x)=\frac{1}{2} g_x(\omega^\sharp,\omega^\sharp)$ is constant. This implies that $\operatorname{Ric}(v)+\Delta f(x)\geqslant0 $ for all $(x,v)\in TM$. Hence, $u$ must be constant by the first part of the proof.
\end{proof}
\appendix
\section{proofs of statements on conjugate points and Jacobi equation}

\begin{proof}[Proof of Proposition \ref{pro:Jabobi_field}]
If $J$ is a Jacobi field along $\rho$, then one can find a variation $\rho(t,s)\in C^2([a, b] \times(-\varepsilon,\varepsilon))$ of extremal curves along $\rho$ such that
\begin{align*}
J(t)=\left.\frac{\partial}{\partial s}\right|_{s=0}\rho(t,s).
\end{align*} 
Thus we have
\begin{align*}
\frac{d}{dt}L_v(\rho(t,s),\dot{\rho}(t,s))=L_x(\rho(t,s),\dot{\rho}(t,s))
\end{align*}
in local chart. Taking partial derivative with respect to $s$ at $s=0$, we get
\begin{align*}
\frac{d}{dt}(L_{vx}(\rho(t),\dot{\rho}(t))J(t)+L_{vv}(\rho(t),\dot{\rho}(t))\dot{J}(t))=L_{xx}(\rho(t),\dot{\rho}(t))J(t)+L_{xv}(\rho(t),\dot{\rho}(t))\dot J(t).
\end{align*}

Now we turn to prove the sufficiency. Suppose $J$ solves \ref{eq:Jacobi} in local chart. We want to define a variation $\rho(t,s)$ such that
\begin{align*}
J(t)=\left.\frac{\partial}{\partial s}\right|_{s=0}\rho(t,s).
\end{align*} 
For this, choose a smooth curve $\alpha:(-\varepsilon,\varepsilon)\to M$ and a smooth vector field $V$ along $\alpha$ satisfying
\begin{gather*}
\alpha(0)=\rho(a),\ \dot\alpha(0)=J(a),\\
V(0)=\dot{\rho}(a),\ \nabla_{\dot\alpha}V(0)=\nabla_{\dot{\rho}}J(a).
\end{gather*}
Now we define
\begin{align*}
\rho(t,s):=\pi\circ\Phi_{t-a}^L(\alpha(s),V(s))\in C^2([a, b]\times(-\varepsilon,\varepsilon)).
\end{align*} 
Then we have
\begin{align*}
\left.\frac{\partial}{\partial t}\right|_{t=a}\rho(t,s)=V(s), \ \left.\frac{\partial}{\partial s}\right|_{s=0}\rho(a,s)=\dot\alpha(0).
\end{align*}
Set
\begin{align*}
W(t):=\left.\frac{\partial}{\partial s}\right|_{s=0}\rho(t,s).
\end{align*}
Note that
\begin{align*}
W(a)=\left.\frac{\partial}{\partial s}\right|_{s=0}\rho(a,s)=\dot\alpha(0)=J(a)
\end{align*}
and $J,W $ solves \ref{eq:Jacobi} in local coordinate systems. If we can show that	$ \nabla_{\dot\rho}W(a)=\nabla_{\dot{\rho}}J(a)$, it then follows from Cauchy-Lipschitz theorem that $J(t)=W(t)$. Indeed, we have
\begin{align*}
\nabla_{\dot\rho}W(a)=\nabla_{\partial_t\rho} \partial_s\rho(a,0)=\nabla_{\partial_s\rho} \partial_t\rho(a,0)=\nabla_{\dot\alpha}V(0)=\nabla_{\dot{\rho}}J(a).
\end{align*} 
It follows that $W\equiv J$ as claimed.   	
\end{proof}

\begin{proof}[Proof of Proposition \ref{pro:Jacobi_formula}]
Consider the $C^2$ variation
\begin{align*}
\rho(s,z)=\pi\circ\Phi_{s-a}^L(\rho(a),\dot{\rho}(a)+z\nabla_{\dot{\rho}(a)}J)\in C^2([a, b]\times(-\varepsilon,\varepsilon)) .
\end{align*}
We have
$$
\begin{gathered}
\left.\partial_z \right|_{z=0} \rho(a, z)=\left.\partial_z \right|_{z=0} \rho(a)=0, \\
\nabla_{\partial_s \rho} \partial_z \rho(a, 0)=\nabla_{  \partial_z \rho        } \partial_s \rho(a, 0)=\left.\partial_z\right|_{z=0} \dot{\rho}(a)+z \nabla_{\dot{\rho}(a)} J=\nabla_{\dot{\rho}(a)}J.
\end{gathered}
$$

Notice that $\left.\partial_z \right|_{z=0}\rho(s,z)$ solves the Jacobi equation \eqref{eq:Jacobi} in local chart and
\begin{align*}
&\,\left.\partial_z\right|_{z=0}\rho(a,z)=J(a)=0,\\
&\,\nabla_{\partial_s\rho}\partial_z\rho(a,0)=\nabla_{\dot{\rho}(a)}J.
\end{align*}
We obtain that
\begin{align*}
J(s)=\left.\partial_z\right|_{z=0}\rho(s,z),\quad s\in[a, b]
\end{align*}
by Cauchy-Lipschitz theorem.
\end{proof} 

\begin{proof}[Proof of Proposition \ref{pro:reverse}]
We only need to prove the necessity since its sufficiency can be proved similarly.

Suppose that $(\rho(b),\dot{\rho}(b))$ is conjugate to $(\rho(b),\dot{\rho}(b))$ with respect to $L$, then there exists a nonzero Jacobi field $J$ along $\rho$ such that $ \rho(a)=\rho(b)=0$. Set $\breve{\rho}(t)=\rho(a+b-t)$, $t\in [a,b]$. Then $(\breve{\rho}(t),\dot{\breve{\rho}}(t))$ is a trajectory of the Euler-Lagrange flow associated to $\breve{L}$. 

Since $J:[a,b]\to TM$ is a Jacobi field along $\rho(t)$ with respect to $L$, one has that $\breve{J}(t):=J(a+b-t)$ is a Jacobi field along $\breve{\rho}(t)$ with respect to $\breve {L}$. Indeed, by a direct computation we have
\begin{align*}
\frac{d}{dt}(\breve{L}_{vx}(\breve{\rho}(t),\dot{\breve{\rho}}(t))\breve{J}(t)+\breve{L}_{vv}(\breve{\rho}(t),\dot{\breve{\rho}}(t))\dot{\breve{J}}(t)) =\breve{L}_{xx}(\breve{\rho}(t),\dot{\breve{\rho}}(t))\breve{J}(t)+\breve{L}_{xv}(\breve{\rho}(t),\dot{\breve{\rho}}(t))\dot{\breve{J}}(t) 
\end{align*}
in local chart. This implies that there is a nonzero Jacobi field $\breve{J}$ along $\breve{\rho}$ with respect to $\breve{L}$ satisfying $\breve{\rho}(a)=\breve{\rho}(b)=0$.    It follows that $(\breve{\rho}(a),\dot {\breve{\rho}}(a))$ is conjugate to $  (\breve{\rho}(b),\dot {\breve{\rho}}(b))$ with respect to $\breve{L}$. In other words, $(\rho(a),-\dot{\rho}(a))$ is conjugate to $(\rho(b),-\dot{\rho}(b))$ with respect to $\breve{L}$.
\end{proof}

\begin{proof}[Proof of Proposition \ref{pro:2nd_var}]

For any variation $\rho(s,z)\in C^{\infty}([0,t]\times(-\varepsilon,\varepsilon)\to M)$ of $\rho(s)$ such that
\begin{align*}
\rho(s,0)=\rho(s), \ \rho(0,z)\equiv\rho(0),\ \rho(t,z)\equiv\rho(t),
\end{align*}
we have
\begin{align*}
\left.\frac{d}{dz}\right|_{z=0}\int_0^tL(\rho,\partial_s\rho)+\omega(\partial_s\rho)ds=\left.\frac{d}{dz}\right|_{z=0}\int_0^tL(\rho,\partial_s\rho)ds,
\end{align*}
since $\omega$ is closed. Therefore,
\begin{align*}
0=&\,\left.\frac{d}{dz}\right|_{z=0}\int_0^tL(\rho,\partial_s\rho)+\omega(\partial_s\rho)ds=\left.\frac{d}{dz}\right|_{z=0}\int_0^tL(\rho,\partial_s\rho)ds\\
=&\,\left.\frac{d}{dz}\right|_{z=0}\int_0^t\frac 12g(\partial_s\rho(s,z),\partial_s\rho(s,z))-f(\rho(s,z))ds\\
=&\,\int_0^tg(\left.\nabla_{\partial_z\rho}\partial_s\rho\right|_{z=0},\dot{\rho})-g(\nabla f(\rho),\left.\partial_z\rho\right|_{z=0})ds\\
=&\,\int_0^tg(\left.\nabla_{\partial_s\rho}\partial_z\rho\right|_{z=0},\dot{\rho})-g(\nabla f(\rho),\left.\partial_z\rho\right|_{z=0})ds\\
=&\,\int_0^t\frac{d}{ds}g(\left.\partial_z\rho\right|_{z=0},\dot{\rho})-g(\left.\partial_z\rho\right|_{z=0},\nabla_{\dot{\rho}}\dot{\rho})-g(\nabla f(\rho),\left.\partial_z\rho\right|_{z=0})ds\\
=&\,\int_0^tg(-\nabla f(\rho)-\nabla_{\dot{\rho}}\dot{\rho},\left.\partial_z \rho\right|_{z=0})ds.
\end{align*}
which yields that
\begin{align*}
\nabla_{\dot{\rho}}\dot{\rho}=-\nabla f.
\end{align*}

Suppose $\rho(s,z)\in C^{2}([a, b]\times(-\varepsilon,\varepsilon),M)$ is a variation such that $\rho(\cdot,z)$, $z\in(-\varepsilon,\varepsilon)$ satisfy \eqref{eq:second_law}. Then we have
\begin{align*}
\nabla_{\partial_s\rho}\nabla_{\partial_s\rho}\partial_z\rho=&\,\nabla_{\partial_s\rho}\nabla_{\partial_z\rho}\partial_s\rho\\
=&\,\nabla_{\partial_z\rho}\nabla_{\partial_s\rho}\partial_s\rho+R(\partial_s\rho,\partial_z\rho)\partial_s\rho\\
=&\,-\nabla_{\partial_z\rho}\nabla f+R(\partial_s\rho,\partial_z\rho)\partial_s\rho\\
=&\,-\operatorname{Hess}f(\partial_z\rho)+R(\partial_s\rho,\partial_z\rho)\partial_s\rho.
\end{align*}
Taking $z=0$ we obtain that
\begin{align*}
\nabla_{\dot{\rho}}\nabla_{\dot{\rho}}J+R(J,\dot{\rho})\dot{\rho}+\operatorname{Hess}f(J)=0, 
\end{align*}
where $J=\left.\partial_z\rho\right|_{z=0}$.
\end{proof}

\begin{proof}[Proof of Lemma \ref{lem:not_conjugate}]
Suppose $J:[a, b]\to TM$ is a Jacobi field along $\rho$ with $J(a)=0$. Set $(x,v)=(\rho(a),\dot{\rho}(a))$, $w=\nabla_{\dot{\rho}(a)}J$. Then we have
\begin{align*}
\left.\partial_z\right|_{z=0}\pi\circ\Phi_{\tau-a}^L(x,v+zw)=d_v(\pi\circ\Phi_{\tau-a}^L)(w)=J(\tau),\qquad \tau\in[a,b].
\end{align*}
Observe that $J$ is nonzero if and only if $w\neq 0$. Therefore, $(\rho(b),\dot{\rho}(b))$ is conjugate to $(\rho(a),\dot{\rho}(a))$ if and only if
\begin{align*}
J(b)=d_v(\pi\circ\Phi_{b-a}^L)(w)=0.
\end{align*}
i.e., if and only if $d_v(\pi\circ\Phi_{b-a}^L)$ is degenerate. Hence, $(\rho(b),\dot{\rho}(b))$ is not conjugate to $(\rho(a),\dot{\rho}(a))$ if and only if $d_v(\pi\circ\Phi_{b-a}^L)$ is non-degenerate.
\end{proof}

\begin{proof}[Proof of Proposition \ref{pro:before_conjugate}]
Since $(\rho(t),\dot{\rho}(t))$ is not conjugate to $(\rho(0),\dot{\rho}(0))$, we can find a neighborhood $W$ of $\dot{\rho}(0)$ such that $\left.\pi\circ\Phi_t^L\right|_W$ is a $C^1$ diffeomorphism.

Now we claim that there exists a neighborhood $U$ of $\rho(t)$ satisfying the following property: if $z\in U$ and $\rho_z$ is a minimizer of $A_t(x,z)$, then we have $\dot{\rho}_z(0)\in W$. If no such neighborhood exists, we can find a sequence $\{z_i\}$ converging to $\rho(t)$ with the following property: for each $z_i$ we can find a minimal curve $\rho_{z_i}$ for $A_t(x, z_i)$ with $\dot{\rho}_{z_i}(0)\notin W$. By (3) of Proposition \ref{pro:differential} we can just suppose $\dot{\rho}_{z_i}(0)\to v_1$ as $i\to+\infty$. By continuous dependence,
\begin{align*}
\int_0^tL(\Phi_s^L(x,\dot{\rho}_{z_i}(0))ds=A_t(x, z_i)
\end{align*}
converges to
\begin{align*}
\int_0^t L(\Phi_s^L(x,v_1))ds=A_t(x,\rho(t))
\end{align*}
as $i\to+\infty$.

This implies that $\Phi_s^L(x,v_1):[0,t]\to M$ is a minimal curve for $A_t(x, \rho(t))$. Hence, $v_1=\dot{\rho}(0)$ since $A_t(x,\cdot)$ is differentiable at $\rho(t)$. But this contradicts the assumption that $\dot{\rho}_{z_i}(0)\notin W$. From this claim we have that for each $z\in U$,
\begin{align*}
A_t(x,z)=\int_0^tL(\Phi_s^L(x,v_z))ds,
\end{align*}
where $v_z$ is uniquely determined by the condition $v_z\in W$, $\pi\circ\Phi_t^L(x,v_z)=z$. In addition, we have $d_yA_t(x,z)=L_v(\Phi_t^L(x,v_z))$ by Proposition \ref{pro:differential}. Therefore. $A_t(x,\cdot)$ is of $C^r$ in $U$ since $U\ni z\mapsto v_z$ is a $ C^{r-1}$ diffeomorphism.
\end{proof}

\begin{proof}[Proof of Lemma \ref{lem:cut}]
Let $\tilde{\rho}(s)=\pi\circ\Phi_s^L(\rho(0),\dot{\rho}(0))$, $s\in[0,+\infty)$. Take a minimal curve $\rho_\tau:[0,t+\tau]\to M$ for $A_{t+\tau}(x,\tilde{\rho}(t+\tau))$ for each fixed $\tau>0$. By (3) of Proposition \ref{pro:differential} we assume that $\dot{\rho}_{\tau_i}(0)$ converges to $v$ as $i\to+\infty$ where $\{\tau_i\}$ is a positive sequence converging to $0$ .

By continuous dependence,
\begin{align*}
\int_0^{t+\tau_i}L( \Phi_s^L(\rho_{\tau_i}(0),\dot{\rho}_{\tau_i}(0)))ds=A_{t+\tau_i}(x,\tilde{\rho}(t+\tau))
\end{align*}
converges to
\begin{align*}
\int_0^tL(\Phi_s^L(x,v)ds=A_t(x,y)
\end{align*}
as $i\to+\infty$.

Then, we can find another minimal curve $\tilde{\rho}(s)=\pi\circ\Phi_s^L(x,v)$, $s\in[0, t]$, for$A_t(x, y)$ if $v\neq\dot{\rho}(0)$. If $\dot{\rho}(0)=v$, we must show that $(\rho(t),\dot{\rho}(t))$ is conjugate to $(\rho(0), \dot{\rho}(0))$. Otherwise, we suppose that $(\rho(t), \dot{\rho}(t))$ is not conjugate to $(\rho(0), \dot{\rho}(0))$. Then, $d_{\dot{\rho}(0)}(\pi\circ \Phi_t^L)$ is non-degenerate. Hence, there exists an open neighborhood $(t-\epsilon,t+\epsilon)\times W $ of $(t,\dot{\rho}(0))$ such that $d_w(\pi\circ\Phi_r^L)$ is non-degenerate for every $(r,w)\in(t-\epsilon,t+\epsilon )\times W$. By constant rank theorem we can know that $\pi\circ\Phi_r^L:W\to M$ is a $C^{r-1}$ diffeomorphism for all $r\in(t-\epsilon, t+\epsilon)$.

Together with
\begin{align*}
\pi\circ\Phi_{t+\tau_i}^L(x,\dot{\rho}_{\tau_i}(0))=\tilde{\rho}(t+\tau_i)=\pi\circ\Phi_{t+\tau_i}^L(x,\dot{\rho}(0)),
\end{align*}
we obtain that $\dot{\rho}(0)=\dot{\rho}_{\tau_i}(0)$ for $\tau_i\in(0,\epsilon)$. This means that $\tilde{\rho}_{\left.\right|_{[0, t+\tau_i]}}$ is a minimal curve for $A_{t+\tau_i}(x, \tilde{\rho}(t+\tau_i))$ which leads to a contradiction.
\end{proof}
\section{Index Form}

It is well known that any two points in the interior of a minimal curve $\rho:[0,t]\to M$ for $A_t(x,y)$, with $\rho(0)=x$ and $\rho(t)=y$, are not conjugate to each other. This is also true for a point $\gamma(t)$ with $t\in(0,t)$ and a endpoint $x=\gamma(0)$ or $y=\gamma(t)$ even if $x$ and $y$ is conjugate to each other  (see, for instance, \cite[Corollary 4.2]{Contreras_Iturriaga1999}). We still give a detailed proof of the form of statement we need in Lagrangian scheme. For the following notion of index  form, see \cite{Duistermaat1976}. 

\begin{defn}
Suppose $L: T M \rightarrow \mathbb{R}$ is a Tonelli Lagrangian defined on an open subset $M$ of $\mathbb{R}^n$. Let $\rho:[a, b] \rightarrow M$ be a $C^1$ curve satisfying \eqref{E-L}. For any two continuous piecewise $C^2$ vector fields $\eta, \theta$ along $\rho$, define the index form $I(\eta, \theta)$ by
\begin{align*}
I(\eta, \theta)=\int_a^b \dot{\eta}^{\top} L_{vv} \dot{\theta}+\dot{\eta}^{\top} L_{v x} \theta+\eta^{\top} L_{xv} \dot{\theta}+\eta^{\top} L_{xx} \theta d t .
\end{align*}
\end{defn} 

\begin{Lem}\label{Index form of Jacobi fields}
Assume that $L: T M \rightarrow \mathbb{R}$ is a Tonelli Lagrangian defined on an open subset $M$ of $\mathbb{R}^n$. Let $\rho:[a, b] \rightarrow M$ be a $C^1$ curve satisfying \eqref{E-L} and let $\eta,\theta$ are continuous piecewise $C^2$ vector fields along $\rho$. For any partition $a=t_0<t_1<\cdots<t_{k+1}=b$ such that the vector fields $\eta, \theta$ along $\rho$ are of $C^2$ on $[t_i, t_{i+1}], i=1,2, \cdots k$, we have
\begin{align*}
I(\eta, \theta)=\sum_{i=1}^{k+1}\left.(\dot{\eta}^{\top} L_{v v}+\eta^{\top} L_{x v})\theta\right|_{t_i^{+}} ^{t_{i+1}^{-}}
\end{align*}
where $\eta|_{[t_i, t_{i+1}]}$ are Jacobi fields along $\rho$. 
\end{Lem}

\begin{proof}
Since $\eta|_{[t_i, t_{i+1}]}$ are Jacobi fields along $\rho$, we have
\begin{align*}
\frac{d}{d t}(L_{v x} \eta+L_{v v} \dot{\eta})=L_{x x} \eta+L_{x v} \dot{\eta}, \quad t \in[t_i, t_{i+1}].
\end{align*}
Therefore
\begin{align*}
I(\eta, \theta)= & \int_a^b \dot{\eta}^{\top} L_{v v} \dot{\theta}+\dot{\eta}^{\top} L_{v x} \theta+\eta^{\top} L_{xv} \dot{\theta}+\eta^{\top} L_{xx }\theta d t \\
= & \int_a^b(\dot{\eta}^{\top} L_{v v}+\eta^{\top} L_{xv})\dot{\theta}+(\dot{\eta}^{\top} L_{v x}+\eta^{\top} L_{x x})\theta d t \\
= & \int_a^b-\frac{d}{d t} (\dot\eta^{\top} L_{v v}+\eta^{\top} L_{x v})\theta+(\dot{\eta}^{\top} L_{v x}+\eta^{\top} L_{x x})\theta d t \\
& +\sum_{i=1}^k\left.(\dot{\eta}^{\top} L_{v v}+\eta^{\top} L_{x v})\theta\right|_{t_i^+} ^{t_{i+1}^-} \\
= & \sum_{i=1}^k\left.(\dot{\eta}^{\top} L_{v v}+\eta^{\top} L_{x v})\theta\right|_{t_i^{+}} ^{    t_{i+1}^-   } .
\end{align*}
This completes the proof.
\end{proof}

\begin{Pro}\label{second variaiton}
Let $L: T M \rightarrow \mathbb{R}$ be a Tonelli Lagrangian on an open subset $M $ of $\mathbb{R}^n$ and let $\rho:[a, b] \rightarrow M$ be a $C^1$ curve satisfying \eqref{E-L}. If
\begin{align*}
\alpha:[a, b] \times(-\varepsilon, \varepsilon)\to&\, M \\
(t, s)\mapsto&\, \alpha(t, s) .
\end{align*}
is a continuous piecewise $C^3$ variation of $\rho$ such that
\begin{enumerate}[\rm (i)]
\item $\alpha(t, 0)=\rho(t), t \in[a, b]$,
\item there exists a partition $a=t_0<t_1<\cdots<t_{k+1}=b$ such that $\alpha$ is of  $C^3$ on $[t_i, t_{i+1}] \times(-\varepsilon, \varepsilon), i=1,2, \cdots, k$.
\end{enumerate}
Then we have
\begin{align*}
\left.\frac{\partial^2}{\partial s^2}\right|_{s=0} \int_a^b L(\alpha, \frac{\partial \alpha}{\partial t})\ dt=I\bigg(\left.\frac{\partial \alpha}{\partial s}\right|_{s=0} , \left.\frac{\partial \alpha}{\partial s}\right|_{s=0}\bigg)+\sum_{i=1}^k \left. L_v \frac{\partial^2 \alpha}{\partial s^2}\right|_{t_i^{+}} ^{t_{i+1}^-} .
\end{align*}
\end{Pro}

\begin{proof}
Notice that
\begin{align*}
\frac{\partial}{\partial s} \int_a^b L(\alpha, \frac{\partial \alpha}{\partial t})d t=\int_a^b L_x \frac{\partial \alpha}{\partial s}+L_v \frac{\partial^2 \alpha}{\partial s \partial t}\, dt.
\end{align*}
Then,
\begin{align*}
\frac{\partial^2}{\partial s^2} \int_a^b L(\alpha, \frac{\partial \alpha}{\partial t})d t= & \int_a^b L_{x x} \frac{\partial \alpha}{\partial s} \cdot \frac{\partial \alpha}{\partial s}+L_{x v} \frac{\partial^2 \alpha}{\partial s \partial t} \cdot \frac{\partial \alpha}{\partial s}+L_{v x} \frac{\partial \alpha}{\partial s} \cdot \frac{\partial^2 \alpha}{\partial s \partial t}+L_{v v} \frac{\partial^2 \alpha}{\partial s \partial t} \cdot \frac{\partial^2 \alpha}{\partial s \partial t} \\
& +L_x \frac{\partial^2 \alpha}{\partial s^2}+L_v \frac{\partial^3 \alpha}{\partial s^2 \partial t} d t \\
= & \int_a^b L_{x x} \frac{\partial \alpha}{\partial s} \cdot \frac{\partial \alpha}{\partial s}+L_{x v} \frac{\partial^2 \alpha}{\partial s \partial t} \cdot \frac{\partial \alpha}{\partial s}+L_{v x} \frac{\partial \alpha}{\partial s} \cdot \frac{\partial^2 \alpha}{\partial s \partial t}+L_{v v} \frac{\partial^2 \alpha}{\partial s \partial t} \cdot \frac{\partial^2 \alpha}{\partial s \partial t} \\
& +L_x \frac{\partial^2 \alpha}{\partial s^2}-\Big(\frac{d}{d t} L_v\Big)\frac{\partial^2 \alpha}{\partial s^2} dt+\sum_{i=1}^k \left. L_v \frac{\partial^2 \alpha}{\partial s^2}\right|_{t_i^{+}} ^{t_{i+1}^-} .
\end{align*}
Taking $s=0$, we obtain
\begin{align*}
\left.\frac{\partial^2}{\partial s^2}\right|_{s=0} \int_a^b L(\alpha, \frac{\partial \alpha}{\partial t})d t=I\bigg(\left.\frac{\partial \alpha}{\partial s}\right|_{s=0} , \left.\frac{\partial \alpha}{\partial s}\right|_{s=0}\bigg)+\sum_{i=1}^k \left. L_v \frac{\partial^2 \alpha}{\partial s^2}\right|_{t_i^{+}} ^{t_{i+1}^-} .
\end{align*}
\end{proof}

\begin{Pro}\label{pro:conjugate}
Suppose $L: T M \rightarrow \mathbb{R}$ is a Tonelli Lagrangian defined on a closed connected Riemannian manifold $(M, g)$. Let $\rho:[a, b] \rightarrow M$ be a $C^1$ curve such that
\begin{align*}
A_{b-a}(\rho(z), \rho(b))=\int_a^b L(\rho(s), \dot{\rho}(s)) ds.
\end{align*}
Then, $(\rho(b), \dot{\rho}(b))$ is not conjugate to $(\rho(z), \dot{\rho}(z))$ 
for each $z \in(a, b)$.
\end{Pro}

\begin{proof}
Since there exists a smooth map $\sigma:[a, b] \times U \rightarrow M$ with $U$ an open neighborhood of $0\in\mathbb{R}^n$, such that $y \mapsto \sigma(t, y)$, mapping 0 to $\rho(t)$, is a diffeomorphism from $U$ to an open neighborhood of $\rho(t)\in M$, the computations can be reduced to the case that $M$ is an open subset of $\mathbb{R}^n$.

Now, suppose that $(\rho(b), \dot{\rho}(b))$ is conjugate to $(\rho(z), \dot{\rho}(z))$ for some $z \in(a, b)$. Then there exists a nonzero Jacobi field $J:[z, b] \rightarrow T M$ along $\rho$ with $J(z)=J(b)=0$. Consider the continuous piecewise $C^2$ vector field
\begin{align*}
\hat{J}(t)=\left\{\begin{array}{cc}
	0, & \text { if } t \in[a, z), \\
	J(t), & \text { if } t \in[z, b],
\end{array}\right.
\end{align*}
and a smooth vector field $E$ along $\rho$ such that $E(a)=E(b)=0$ and $E(z)=\dot J(z) \neq 0$. Set $V(t)=\hat J(t)+\lambda E(t),\ t \in[a, b]$ and
\begin{align*}
\alpha(t, s)=\rho(t)+s V(t), \quad(t, s) \in[a, b] \times(-\varepsilon, \varepsilon) .
\end{align*}
Then we have $\alpha(t, 0)=\rho(t), \alpha(a, s)=\rho(a), \alpha(b, s)=\rho(b)$ and $\frac{\partial \alpha}{\partial s}=V$. For the action
\begin{align*}
A_\alpha(s):=\int_a^b L(\alpha(t, s), \frac{\partial \alpha}{\partial t}(t, s))d t,
\end{align*}
we obtain that
\begin{align*}
\frac{d}{ds}	A_\alpha(0) & =\int_a^b L_x \frac{\partial \alpha}{\partial s}+L_v \frac{\partial^2 \alpha}{\partial s \partial t} d t \\
& =\int_a^b L_x \frac{\partial \alpha}{\partial s}-\big(\frac{d}{d t} L_v\big)\frac{\partial \alpha}{\partial s} d t+\left.L_v \frac{\partial \alpha}{\partial s}\right|_a ^{z^-} +\left.L_v \frac{\partial \alpha}{\partial s}\right|_{z^+} ^b=0.
\end{align*}
Moreover, by Proposition \ref{second variaiton} and Lemma \ref{Index form of Jacobi fields} one has
\begin{align*}
\frac{d^2}{ds^2}	A_\alpha(0) & =I\big(\left.\frac{\partial \alpha}{\partial s}\right|_{s=0} , \left.\frac{\partial \alpha}{\partial s}\right|_{s=0}\big)+\left.L_v \frac{\partial^2 \alpha}{\partial s^2}\right|_a ^{z^-} +\left.L_v \frac{\partial^2 \alpha}{\partial s^2}\right|_{z^+} ^b=I(V, V) \\
& =I(\hat J+\lambda E, \hat J+\lambda E) \\
& =I(\hat J, \hat J)+2 \lambda I(\hat J, E)+\lambda^2 I(E, E) \\
& =2 \lambda I(\hat J, E)+\lambda^2 I(E, E) \\
& =-2 \lambda(\dot J^{\top}(z) L_{v v}\dot J(z))+\lambda^2 I(E, E) .
\end{align*}
Therefore, $\ddot{\rho}(0)<0$ if $0<\lambda<2 \dot J^{\top}(z) L_{v v}\dot J(z)  / |I(E, E)|$. This implies that $\rho_{\mid[a, b]}$ is not minimizing. This leads to a contradiction.  
\end{proof}

With similar argument one also has the following statement.

\begin{Cor}
Suppose $L: T M \rightarrow \mathbb{R}$ is a Tonelli Lagrangian defined on a closed connected Riemannian manifold $(M, g)$. Let $\rho:[a, b] \to M$ be a $C^1$ curve such that
\begin{align*}
A_{b-a}(\rho(a), \rho(z))=\int_a^b L(\rho(s), \dot{\rho}(s))\ ds.
\end{align*}
Then, $(\rho(z), \dot{\rho}(z))$ is not conjugate to $(\rho(a), \dot{\rho}(a))$ for each $z \in(a, b)$. 
\end{Cor}

\bibliographystyle{plain}
\bibliography{mybib}

\begin{thebibliography}{10}

\bibitem{Bernard2010b}
Patrick Bernard.
\newblock On the {C}onley decomposition of {M}ather sets.
\newblock {\em Rev. Mat. Iberoam.}, 26(1):115--132, 2010.

\bibitem{Bernard2012}
Patrick Bernard.
\newblock The {L}ax-{O}leinik semi-group: a {H}amiltonian point of view.
\newblock {\em Proc. Roy. Soc. Edinburgh Sect. A}, 142(6):1131--1177, 2012.

\bibitem{Bernard_Contreras2008}
Patrick Bernard and Gonzalo Contreras.
\newblock A generic property of families of {L}agrangian systems.
\newblock {\em Ann. of Math. (2)}, 167(3):1099--1108, 2008.

\bibitem{Burago_Ivanov_Kleiner1997}
D.~Burago, S.~Ivanov, and B.~Kleiner.
\newblock On the structure of the stable norm of periodic metrics.
\newblock {\em Math. Res. Lett.}, 4(6):791--808, 1997.

\bibitem{Calabi1958}
E.~Calabi.
\newblock An extension of {E}. {H}opf's maximum principle with an application
  to {R}iemannian geometry.
\newblock {\em Duke Math. J.}, 25:45--56, 1958.

\bibitem{Cannarsa_Cheng3}
Piermarco Cannarsa and Wei Cheng.
\newblock Generalized characteristics and {L}ax-{O}leinik operators: global
  theory.
\newblock {\em Calc. Var. Partial Differential Equations}, 56(5):Art. 125, 31,
  2017.

\bibitem{Cannarsa_Sinestrari_book}
Piermarco Cannarsa and Carlo Sinestrari.
\newblock {\em Semiconcave functions, {H}amilton-{J}acobi equations, and
  optimal control}, volume~58 of {\em Progress in Nonlinear Differential
  Equations and their Applications}.
\newblock Birkh{\"a}user Boston, Inc., Boston, MA, 2004.

\bibitem{Contreras_Iturriaga1999}
Gonzalo Contreras and Renato Iturriaga.
\newblock Convex {H}amiltonians without conjugate points.
\newblock {\em Ergodic Theory Dynam. Systems}, 19(4):901--952, 1999.

\bibitem{Contreras_Iturriagabook1999}
Gonzalo Contreras and Renato Iturriaga.
\newblock {\em Global minimizers of autonomous {L}agrangians}.
\newblock 22$^{\rm o}$ Col{\'o}quio Brasileiro de Matem{\'a}tica. Instituto de
  Matem{\'a}tica Pura e Aplicada (IMPA), Rio de Janeiro, 1999.

\bibitem{Contreras_Miranda2020}
Gonzalo Contreras and Jos\'{e} Ant\^{o}nio~G. Miranda.
\newblock On finite quotient {A}ubry set for generic geodesic flows.
\newblock {\em Math. Phys. Anal. Geom.}, 23(2):Paper No. 14, 11, 2020.

\bibitem{do_Carmo1992book}
Manfredo Perdig\~{a}o do~Carmo.
\newblock {\em Riemannian geometry}.
\newblock Mathematics: Theory \& Applications. Birkh\"{a}user Boston, Inc.,
  Boston, MA, 1992.

\bibitem{Duistermaat1976}
J.~J. Duistermaat.
\newblock On the {M}orse index in variational calculus.
\newblock {\em Advances in Math.}, 21(2):173--195, 1976.

\bibitem{Eschenburg_Heintze1984}
Jost Eschenburg and Ernst Heintze.
\newblock An elementary proof of the {C}heeger-{G}romoll splitting theorem.
\newblock {\em Ann. Global Anal. Geom.}, 2(2):141--151, 1984.

\bibitem{Fathi_book}
Albert Fathi.
\newblock {W}eak {KAM} theorem in {L}agrangian dynamics.
\newblock Cambridge University Press, Cambridge (to appear).

\bibitem{Fathi_Figalli2010}
Albert Fathi and Alessio Figalli.
\newblock Optimal transportation on non-compact manifolds.
\newblock {\em Israel J. Math.}, 175:1--59, 2010.

\bibitem{Fathi_Figalli_Rifford2009}
Albert Fathi, Alessio Figalli, and Ludovic Rifford.
\newblock On the {H}ausdorff dimension of the {M}ather quotient.
\newblock {\em Comm. Pure Appl. Math.}, 62(4):445--500, 2009.

\bibitem{Fathi_Siconolfi2004}
Albert Fathi and Antonio Siconolfi.
\newblock Existence of {$C^1$} critical subsolutions of the {H}amilton-{J}acobi
  equation.
\newblock {\em Invent. Math.}, 155(2):363--388, 2004.

\bibitem{Mane1992}
Ricardo Ma\~n{\'e}.
\newblock On the minimizing measures of {L}agrangian dynamical systems.
\newblock {\em Nonlinearity}, 5(3):623--638, 1992.

\bibitem{Mather1991}
John~N. Mather.
\newblock Action minimizing invariant measures for positive definite
  {L}agrangian systems.
\newblock {\em Math. Z.}, 207(2):169--207, 1991.

\bibitem{Mather1993}
John~N. Mather.
\newblock Variational construction of connecting orbits.
\newblock {\em Ann. Inst. Fourier (Grenoble)}, 43(5):1349--1386, 1993.

\bibitem{Mather2003}
John~N. Mather.
\newblock Total disconnectedness of the quotient {A}ubry set in low dimensions.
\newblock volume~56, pages 1178--1183. 2003.
\newblock Dedicated to the memory of J\"{u}rgen K. Moser.

\bibitem{Mather2004}
John~N. Mather.
\newblock Examples of {A}ubry sets.
\newblock {\em Ergodic Theory Dynam. Systems}, 24(5):1667--1723, 2004.

\bibitem{Petersen_book2016}
Peter Petersen.
\newblock {\em Riemannian geometry}, volume 171 of {\em Graduate Texts in
  Mathematics}.
\newblock Springer, Cham, third edition, 2016.

\bibitem{Rifford2008}
Ludovic Rifford.
\newblock On viscosity solutions of certain {H}amilton-{J}acobi equations:
  regularity results and generalized {S}ard's theorems.
\newblock {\em Comm. Partial Differential Equations}, 33(1-3):517--559, 2008.

\bibitem{Sakai1996book}
Takashi Sakai.
\newblock {\em Riemannian geometry}, volume 149 of {\em Translations of
  Mathematical Monographs}.
\newblock American Mathematical Society, Providence, RI, 1996.

\bibitem{Sorrentino2008}
Alfonso Sorrentino.
\newblock On the total disconnectedness of the quotient {A}ubry set.
\newblock {\em Ergodic Theory Dynam. Systems}, 28(1):267--290, 2008.

\bibitem{Villani_book2009}
C\'{e}dric Villani.
\newblock {\em Optimal transport: old and new}, volume 338 of {\em Grundlehren
  der Mathematischen Wissenschaften}.
\newblock Springer-Verlag, Berlin, 2009.

\bibitem{Warner1983book}
Frank~W. Warner.
\newblock {\em Foundations of differentiable manifolds and {L}ie groups},
  volume~94 of {\em Graduate Texts in Mathematics}.
\newblock Springer-Verlag, New York-Berlin, 1983.

\bibitem{Wei_Wylie2009}
Guofang Wei and Will Wylie.
\newblock Comparison geometry for the {B}akry-{E}mery {R}icci tensor.
\newblock {\em J. Differential Geom.}, 83(2):377--405, 2009.

\end{thebibliography}

\end{document}